\documentclass[10pt]{article}

\usepackage{definitions}

\usepackage{amsmath,amssymb,mathrsfs,amsthm}
\usepackage{bm}
\usepackage{longtable}
\usepackage[ruled,linesnumbered]{algorithm2e}
\usepackage{natbib}
\setcitestyle{authoryear,round}

\newtheorem{theorem}{Theorem}
\newtheorem{prop}[theorem]{Proposition}
\newtheorem{example}[theorem]{Example}

\newtheorem{definition}[theorem]{Definition}

\usepackage{xcolor}

\SetCommentSty{mycommfont}

\title{A New Crossover Algorithm for LP Inspired by the Spiral Dynamic of PDHG}
\author{Tianhao Liu\thanks{Shanghai Jiao Tong University, Antai College of Economics and Management (tianhao.liu@sjtu.edu.cn).} \and Haihao Lu\thanks{MIT, Sloan School of Management (haihao@mit.edu).}}
\date{}

\begin{document}
\global\long\def\inprod#1#2{\left\langle #1,#2\right\rangle }%
\global\long\def\inner#1#2{\langle#1,#2\rangle}%
\global\long\def\binner#1#2{\big\langle#1,#2\big\rangle}%
\global\long\def\norm#1{\Vert#1\Vert}%
\global\long\def\bnorm#1{\big\Vert#1\big\Vert}%
\global\long\def\Bnorm#1{\Big\Vert#1\Big\Vert}%
\global\long\def\red#1{\textcolor{red}{#1}}%
\global\long\def\blue#1{\textcolor{blue}{#1}}%
\global\long\def\KKT{\text{KKT}}%

\global\long\def\brbra#1{\big(#1\big)}%
\global\long\def\Brbra#1{\Big(#1\Big)}%
\global\long\def\rbra#1{(#1)}%

\global\long\def\sbra#1{[#1]}%
\global\long\def\bsbra#1{\big[#1\big]}%
\global\long\def\Bsbra#1{\Big[#1\Big]}%
\global\long\def\abs#1{\vert#1\vert}%
\global\long\def\babs#1{\big\vert#1\big\vert}%

\global\long\def\cbra#1{\{#1\}}%
\global\long\def\bcbra#1{\big\{#1\big\}}%
\global\long\def\Bcbra#1{\Big\{#1\Big\}}%
\global\long\def\vertiii#1{\left\vert \kern-0.25ex  \left\vert \kern-0.25ex  \left\vert #1\right\vert \kern-0.25ex  \right\vert \kern-0.25ex  \right\vert }%
\global\long\def\matr#1{\bm{#1}}%
\global\long\def\til#1{\tilde{#1}}%
\global\long\def\wtil#1{\widetilde{#1}}%
\global\long\def\wh#1{\widehat{#1}}%
\global\long\def\mcal#1{\mathcal{#1}}%
\global\long\def\mbb#1{\mathbb{#1}}%
\global\long\def\mtt#1{\mathtt{#1}}%
\global\long\def\ttt#1{\texttt{#1}}%
\global\long\def\dtxt{\textrm{d}}%
\global\long\def\bignorm#1{\bigl\Vert#1\bigr\Vert}%
\global\long\def\Bignorm#1{\Bigl\Vert#1\Bigr\Vert}%
\global\long\def\rmn#1#2{\mathbb{R}^{#1\times#2}}%
\global\long\def\deri#1#2{\frac{d#1}{d#2}}%
\global\long\def\pderi#1#2{\frac{\partial#1}{\partial#2}}%
\global\long\def\limk{\lim_{k\rightarrow\infty}}%
\global\long\def\trans{\textrm{T}}%
\global\long\def\onebf{\mathbf{1}}%
\global\long\def\oneds{\mathds{1}}%
\global\long\def\Bbb{\mathbb{B}}%
\global\long\def\hbf{\mathbf{h}}%
\global\long\def\zerobf{\mathbf{0}}%
\global\long\def\zero{\bm{0}}%

\global\long\def\Expe{\mathbb{E}}%
\global\long\def\inte{\operatornamewithlimits{int}}%
\global\long\def\dist{\operatornamewithlimits{dist}}%
\global\long\def\proj{\operatorname{Proj}}%
\global\long\def\argmin{\operatornamewithlimits{argmin}}%
\global\long\def\argmax{\operatornamewithlimits{argmax}}%
\global\long\def\tr{\operatornamewithlimits{tr}}%
\global\long\def\dis{\operatornamewithlimits{dist}}%
\global\long\def\sign{\operatornamewithlimits{sign}}%
\global\long\def\st{\operatornamewithlimits{s.t.}}%
\global\long\def\aleq{\overset{(a)}{\leq}}%
\global\long\def\aeq{\overset{(a)}{=}}%
\global\long\def\ageq{\overset{(a)}{\geq}}%
\global\long\def\bleq{\overset{(b)}{\leq}}%
\global\long\def\beq{\overset{(b)}{=}}%
\global\long\def\bgeq{\overset{(b)}{\geq}}%
\global\long\def\cleq{\overset{(c)}{\leq}}%
\global\long\def\ceq{\overset{(c)}{=}}%
\global\long\def\cgeq{\overset{(c)}{\geq}}%
\global\long\def\dleq{\overset{(d)}{\leq}}%
\global\long\def\deq{\overset{(d)}{=}}%
\global\long\def\dgeq{\overset{(d)}{\geq}}%
\global\long\def\eleq{\overset{(e)}{\leq}}%
\global\long\def\eeq{\overset{(e)}{=}}%
\global\long\def\egeq{\overset{(e)}{\geq}}%
\global\long\def\fleq{\overset{(f)}{\leq}}%
\global\long\def\feq{\overset{(f)}{=}}%
\global\long\def\fgeq{\overset{(f)}{\geq}}%
\global\long\def\gleq{\overset{(g)}{\leq}}%
\global\long\def\as{\textup{a.s.}}%
\global\long\def\ae{\textup{a.e.}}%
\global\long\def\Var{\operatornamewithlimits{Var}}%
\global\long\def\clip{\operatorname{clip}}%
\global\long\def\conv{\operatorname{conv}}%
\global\long\def\Cov{\operatornamewithlimits{Cov}}%
\global\long\def\raw{\rightarrow}%
\global\long\def\law{\leftarrow}%
\global\long\def\Raw{\Rightarrow}%
\global\long\def\Law{\Leftarrow}%
\global\long\def\vep{\varepsilon}%
\global\long\def\dom{\operatornamewithlimits{dom}}%
\global\long\def\tsum{{\textstyle {\sum}}}%
\global\long\def\Cbb{\mathbb{C}}%
\global\long\def\Ebb{\mathbb{E}}%
\global\long\def\Fbb{\mathbb{F}}%
\global\long\def\Nbb{\mathbb{N}}%
\global\long\def\Rbb{\mathbb{R}}%
\global\long\def\Sbb{\mathbb{S}}%
\global\long\def\extR{\widebar{\mathbb{R}}}%
\global\long\def\Pbb{\mathbb{P}}%
\global\long\def\Zbb{\mathbb{Z}}%
\global\long\def\Acal{\mathcal{A}}%
\global\long\def\Bcal{\mathcal{B}}%
\global\long\def\Ccal{\mathcal{C}}%
\global\long\def\Dcal{\mathcal{D}}%
\global\long\def\Ecal{\mathcal{E}}%
\global\long\def\Fcal{\mathcal{F}}%
\global\long\def\Gcal{\mathcal{G}}%
\global\long\def\Hcal{\mathcal{H}}%
\global\long\def\Ical{\mathcal{I}}%
\global\long\def\Jcal{\mathcal{J}}%
\global\long\def\Kcal{\mathcal{K}}%
\global\long\def\Lcal{\mathcal{L}}%
\global\long\def\Mcal{\mathcal{M}}%
\global\long\def\Ncal{\mathcal{N}}%
\global\long\def\Ocal{\mathcal{O}}%
\global\long\def\Pcal{\mathcal{P}}%
\global\long\def\Scal{\mathcal{S}}%
\global\long\def\Tcal{\mathcal{T}}%
\global\long\def\Xcal{\mathcal{X}}%
\global\long\def\Ycal{\mathcal{Y}}%
\global\long\def\Zcal{\mathcal{Z}}%
\global\long\def\i{i}%

\global\long\def\abf{\mathbf{a}}%
\global\long\def\bbf{\mathbf{b}}%
\global\long\def\cbf{\mathbf{c}}%
\global\long\def\dbf{\mathbf{d}}%
\global\long\def\ebf{\mathbf{e}}%
\global\long\def\fbf{\mathbf{f}}%
\global\long\def\gbf{\mathbf{g}}%
\global\long\def\hbf{\mathbf{h}}%
\global\long\def\ibf{\mathbf{i}}%
\global\long\def\jbf{\mathbf{j}}%
\global\long\def\kbf{\mathbf{k}}%
\global\long\def\lbf{\mathbf{l}}%
\global\long\def\mbf{\mathbf{m}}%
\global\long\def\nbf{\mathbf{n}}%
\global\long\def\obf{\mathbf{o}}%
\global\long\def\pbf{\mathbf{p}}%
\global\long\def\qbf{\mathbf{q}}%
\global\long\def\rbf{\mathbf{r}}%
\global\long\def\sbf{\mathbf{s}}%
\global\long\def\tbf{\mathbf{t}}%
\global\long\def\ubf{\mathbf{u}}%
\global\long\def\vbf{\mathbf{v}}%
\global\long\def\wbf{\mathbf{w}}%
\global\long\def\xbf{\mathbf{x}}%
\global\long\def\ybf{\mathbf{y}}%
\global\long\def\zbf{\mathbf{z}}%
\global\long\def\lambf{\bm{\lambda}}%
\global\long\def\alphabf{\bm{\alpha}}%
\global\long\def\sigmabf{\bm{\sigma}}%
\global\long\def\Sigmabf{\bm{\Sigma}}%
\global\long\def\thetabf{\bm{\theta}}%
\global\long\def\deltabf{\bm{\delta}}%
\global\long\def\Deltabf{\bm{\Delta}}%
\global\long\def\Omegabf{\bm{\Omega}}%
\global\long\def\Lambf{\bm{\Lambda}}%
\global\long\def\pibf{\bm{\pi}}%
\global\long\def\Wbf{\mathbf{W}}%
\global\long\def\Abf{\mathbf{A}}%
\global\long\def\Jbf{\mathbf{J}}%
\global\long\def\Ubf{\mathbf{U}}%
\global\long\def\Vbf{\mathbf{V}}%
\global\long\def\Pbf{\mathbf{P}}%
\global\long\def\Ibf{\mathbf{I}}%
\global\long\def\Ebf{\mathbf{E}}%
\global\long\def\Mbf{\mathbf{M}}%
\global\long\def\Nbf{\mathbf{N}}%
\global\long\def\Dbf{\mathbf{D}}%
\global\long\def\Qbf{\mathbf{Q}}%
\global\long\def\Lbf{\mathbf{L}}%
\global\long\def\Pbf{\mathbf{P}}%
\global\long\def\Xbf{\mathbf{X}}%
\global\long\def\Ybf{\mathbf{Y}}%
\global\long\def\Bbf{\mathbf{B}}%
\global\long\def\Tbf{\mathbf{T}}%
\global\long\def\Hbf{\mathbf{H}}%
\global\long\def\Sbf{\mathbf{S}}%


\global\long\def\abm{\bm{a}}%
\global\long\def\bbm{\bm{b}}%
\global\long\def\cbm{\bm{c}}%
\global\long\def\dbm{\bm{d}}%
\global\long\def\ebm{\bm{e}}%
\global\long\def\fbm{\bm{f}}%
\global\long\def\gbm{\bm{g}}%
\global\long\def\hbm{\bm{h}}%
\global\long\def\pbm{\bm{p}}%
\global\long\def\qbm{\bm{q}}%
\global\long\def\rbm{\bm{r}}%
\global\long\def\sbm{\bm{s}}%
\global\long\def\tbm{\bm{t}}%
\global\long\def\ubm{\bm{u}}%
\global\long\def\vbm{\bm{v}}%
\global\long\def\wbm{\bm{w}}%
\global\long\def\xbm{\bm{x}}%
\global\long\def\ybm{\bm{y}}%
\global\long\def\zbm{\bm{z}}%
\global\long\def\Abm{\bm{A}}%
\global\long\def\Bbm{\bm{B}}%
\global\long\def\Cbm{\bm{C}}%
\global\long\def\Dbm{\bm{D}}%
\global\long\def\Ebm{\bm{E}}%
\global\long\def\Fbm{\bm{F}}%
\global\long\def\Gbm{\bm{G}}%
\global\long\def\Hbm{\bm{H}}%
\global\long\def\Ibm{\bm{I}}%
\global\long\def\Jbm{\bm{J}}%
\global\long\def\Lbm{\bm{L}}%
\global\long\def\Obm{\bm{O}}%
\global\long\def\Pbm{\bm{P}}%
\global\long\def\Qbm{\bm{Q}}%
\global\long\def\Rbm{\bm{R}}%
\global\long\def\Ubm{\bm{U}}%
\global\long\def\Vbm{\bm{V}}%
\global\long\def\Wbm{\bm{W}}%
\global\long\def\Xbm{\bm{X}}%
\global\long\def\Ybm{\bm{Y}}%
\global\long\def\Zbm{\bm{Z}}%
\global\long\def\lambm{\bm{\lambda}}%
\global\long\def\alphabm{\bm{\alpha}}%
\global\long\def\albm{\bm{\alpha}}%
\global\long\def\taubm{\bm{\tau}}%
\global\long\def\mubm{\bm{\mu}}%
\global\long\def\inftybm{\bm{\infty}}%

\maketitle

\begin{abstract}
    Motivated by large-scale applications, there is a recent trend of research on using first-order methods for solving LP. Among them, PDLP, which is based on a primal-dual hybrid gradient (PDHG) algorithm, may be the most promising one. In this paper, we present a geometric viewpoint on the behavior of PDHG for LP. We demonstrate that PDHG iterates exhibit a spiral pattern with a closed-form solution when the variable basis remains unchanged. This spiral pattern consists of two orthogonal components: rotation and forward movement, where rotation improves primal and dual feasibility, while forward movement advances the duality gap. We also characterize the different situations in which basis change events occur. Inspired by the spiral behavior of PDHG, we design a new crossover algorithm to obtain a vertex solution from any optimal LP solution. This approach differs from traditional simplex-based crossover methods. Our numerical experiments demonstrate the effectiveness of the proposed algorithm, showcasing its potential as an alternative option for crossover.
\end{abstract}

\section{Introduction} \label{sec:intro}
Linear programming (LP), which refers to optimizing a linear function over a polyhedron, is one of the most fundamental classes of optimization problems. It has been extensively studied in both academia and industry and has been applied to solve real-world applications in transportation, scheduling, economy, resource allocation, etc.

Two classic methods for solving LP are simplex methods and interior point methods (IPMs). The simplex methods, proposed by Dantzig in the late 1940s \citep{dantzig1951maximization}, start from a vertex and pivot between vertices to monotonically improve the objective \citep{dantzig1963linear}. IPMs \citep{karmarkar1984new}, on the other hand, utilize a self-concordant barrier function to ensure that solutions remain inside the polyhedron, approaching an optimal solution by following a central path  \citep{renegar1988polynomial}.

Motivated by applications of large-scale LP, there has been a recent trend towards developing first-order methods (FOMs) for LP, such as PDLP \citep{applegate2021practical,lu2023cupdlp,lu2023cupdlpc}, SCS \citep{ocpb:16, odonoghue:21}, and ABIP \citep{lin2021admm, deng2024enhanced}.  The key advantage of FOMs is their low iteration cost, primarily involving matrix-vector multiplications, which avoids the need for solving linear equations as required in the simplex methods and IPMs. This advantage makes them particularly well-suited for GPU implementation. Among these, PDLP stands out as the most promising, with its GPU implementation demonstrating numerical performance on par with state-of-the-art LP solvers on standard benchmark sets and demonstrating superior performance on large-scale instances~\citep{lu2023cupdlp,lu2023cupdlpc}. PDLP is based on the primal-dual hybrid gradient (PDHG) method, enhanced by various numerical improvements~\citep{applegate2021practical}. Variants of PDLP have been implemented in both commercial and open-source optimization solvers, such as COPT \citep{copt}, Xpress \citep{xpress2014fico}, Google OR-Tools \citep{ortools}, HiGHS \citep{huangfu2018parallelizing}, and NVIDIA cuOpt \citep{cuopt}.

While the behaviors of simplex methods and IPMs are well-studied---simplex methods move over vertices, and IPMs follow the central path---the convergence behaviors of PDLP and other FOMs toward an optimal solution are less well-understood. For example, we plot the 2-dimensional primal trajectories \footnote{For the LP $\min_{\xbf\in\Rbb_+^2}\{ \cbf^\top\xbf : \Abf\xbf \leq \bbf \}$, we reformulate it as $\min_{(\xbf,\wbf)\in\Rbb_+^2\times\Rbb_+^m}\{ \cbf^\top\xbf : \Abf\xbf + \wbf = \bbf \}$ and plot the trajectories of $\xbf$ from different methods solving the reformulated problem.} of different LP methods in Figure \ref{fig:lp-methods}. PDHG seems to progress in a more chaotic and winding manner compared to other classic methods.
\begin{figure}[!ht]
    \centering
    \includegraphics[width=0.5\linewidth]{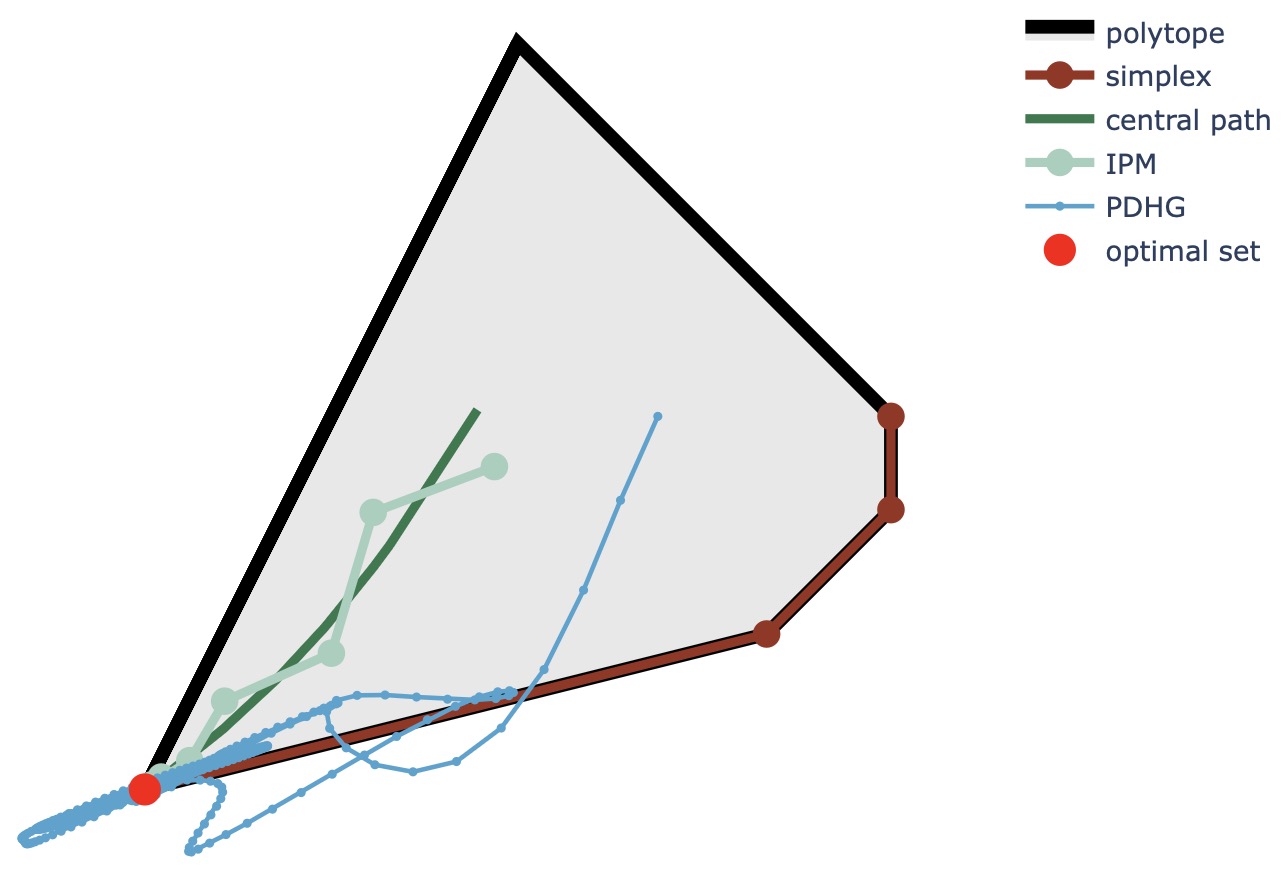}
    \caption{Primal trajectories of different LP methods.}
    \label{fig:lp-methods}
\end{figure}

The first contribution of this paper is to present a geometric viewpoint on the convergence behavior of the PDHG algorithm when applied to LP problems (Section \ref{sec:pdlp}). Similar to simplex methods, we can define basic and non-basic variables for PDHG iterates. We define a phase as a series of consecutive PDHG iterates that share the same set of basic variables. Within each phase, we demonstrate that PDHG behaves like a spiral, consisting of two orthogonal components: rotation and forward movement. The rotation improves primal and dual feasibility, while the forward movement advances the duality gap. We also characterize the different situations in which basis change events occur.

A fundamental distinction among different LP methods is that simplex methods identify an optimal vertex (i.e., an optimal basis), whereas IPMs and FOMs may find an interior point within the optimal solution set. Specifically, IPMs converge to the analytic center of the optimal solution set, while FOMs can reach any optimal solution depending on the initial solution. On the other hand, vertex solutions obtained via simplex methods are generally more precise, sparse, and informative. For example, vertex solutions are particularly useful for LP subroutines in branch-and-bound MIP solvers, as they enhance integrality and also benefit warm starts. Consequently, the process of ``crossover,'' which involves deriving an optimal vertex from any arbitrary optimal solution, has been extensively studied for IPMs. Most modern LP solvers include a crossover step following IPM solving. The algorithmic framework for contemporary crossover was introduced by \citet{megiddo1991finding} and has been further refined by \citet{bixby1994recovering} and \citet{andersen1996combining}. These approaches, which rely on simplex methods to adjust variables to their bounds, are referred to as simplex-based crossover.

Inspired by the spiral behavior of PDHG, our second contribution is the design of a new crossover method to generate a vertex solution from a PDLP solution. This crossover leverages PDLP itself and its spiral axes (referred to as spiral rays in this paper), eliminating the need for any simplex steps. We present a numerical study that demonstrates the effectiveness of our proposed crossover algorithm, which may provide a practical alternative to simplex-based crossover methods.

\paragraph{Organization of the paper.}

The paper is organized as follows. Section~\ref{sec:rel-work} reviews related studies in PDHG and crossover approaches. Section~\ref{sec:pdlp} presents the spiral behavior of PDHG for LP. Section~\ref{sec:cross} proposes a new crossover algorithm inspired by the spiral structure of PDHG without the need for simplex methods. Section~\ref{sec:exp} presents a numerical study of our crossover approach.

\paragraph{Notations.}

We use bold letters for vectors and matrices. Let $\xbf_j$ be the $j$-th component of vector $\xbf$, and $\xbf_\Jcal$ be the vector formed by components $\xbf_{j}$ for $j\in\Jcal$. Let $\Abf_{j}$ be the $j$-th column of matrix $\Abf$, and $\Abf_\Jcal$ be the submatrix formed by columns $\Abf_{j}$ for $j\in\Jcal$. Let $\Abf_{i,j}$ be the entry in the $i$-th row and $j$-th column of matrix $\Abf$, and $\Abf_{\Ical, \Jcal}$ be the submatrix formed by entries $\Abf_{i,j}$ for $i\in\Ical, j\in\Jcal$. We use $\xbf \geq \ybf$ to express the element-wise inequality $\xbf_{i} \geq \ybf_{i}$ for all $i$. Let $\zerobf$ and $\onebf$ be a vector or matrix of zeros and ones, respectively. Let $\Ibf$ be the identity matrix. The dimension of a vector or a matrix will be unspecified whenever it is clear from the context. Let $\Rbb^n$ denote the $n$-dimensional Euclidean space and $\Rbb_+^n = \{\xbf\in\Rbb^n : \xbf \geq \zerobf\}$. The vector norm $\|\cdot\|_\ell$ is $\ell$-norm. The $\Mbf$-norm $\|\vbf\|_\Mbf = \sqrt{\vbf^\top \Mbf \vbf}$ where $\Mbf$ is a positive definite matrix. The matrix norm $\|\cdot\|_2$ is the spectral norm. $\xbf^{(k)}$ means the $k$-th point $\xbf$, while $\Pbf^k$ without the parenthesis means $\Pbf$ to the power of $k$. $\Abf^\dagger$ is the Moore–Penrose inverse of $\Abf$. The operator $\proj_{\Omegabf}(\xbf)$ projects $\xbf$ onto the set $\Omegabf$ using 2-norm.
The range of an operator $\Tbf$ is denoted as $\text{range}(\Tbf)$ while the closure of a set $S$ is denoted as $\text{cl}(S)$.

\subsection{Related Work} \label{sec:rel-work}
\paragraph{PDLP and related.}
PDHG was originally proposed in~\cite{zhu2008efficient, esser2010general}, and further
studied in \citet{chambolle2011first} for image processing applications. \citet{applegate2021practical} developed a restarted PDHG algorithm for solving linear programs (LPs), which, together with a few practical enhancements, leads to the PDLP solver. Compared with other popular FOM solvers such as SCS and ABIP, there is no linear system to solve in PDLP, making it a more suitable choice for large-scale LPs. cuPDLP, the GPU implementation of PDLP, demonstrated strong numerical performance on par with commercial LP solvers, and has attracted attention from both the academic and solver industry \citep{lu2023cupdlp, lu2023cupdlpc}. Recently, additional numerical enhancements have been proposed to speed up the algorithm further. \citet{xiong2024role} pointed out that the central-path-based scaling can markedly improve the convergence rate of PDHG. \citet{lu2024restarted} embedded Halpern iteration \citep{halpern1967fixed} in PDLP, achieving accelerated theoretical guarantees and improved computational results.

The practical success of PDLP motivates a stream of theoretical analysis. \citet{lu2022infimal} showed PDHG has a linear convergence rate when solving LP. \citet{applegate2023faster} revealed that an accelerated linear convergence rate can be achieved with restarts, and such an accelerated linear rate matches the lower bound. \citet{applegate2024infeasibility} illustrated that for infeasible/unbounded LP, PDHG iterates diverge towards the direction of infeasibility/unboundedness certificate; thus, one can detect the infeasibility/unboundedness of LP using PDHG without additional effort. \citet{lu2024geometry} discovered that PDHG has a two-stage behavior when solving LP. In the first stage, PDHG uses finite iterations to verify the optimal active variables; in the second stage, PDHG solves a homogeneous linear inequality system with a significantly improved linear rate. This explains the slow behaviors of PDHG when the LP is near-degeneracy. \citet{xiong2024role} introduced level-set geometry that characterizes the difficulty of PDHG for LP, which shows that PDHG converges faster for LP with a regular non-flat level-set around the optimal solution set. 

For more general fixed point problems, trajectories of various FOMs were studied in \citet{poon2019trajectory, poon2020geometry}, which showed that eventually (after a finite number of iterations), these algorithms follow a straight line or a spiral structure if the problem is non-degenerate. Unfortunately, this non-degeneracy assumption is generally never satisfied for real-world LP instances.

\paragraph{Crossover.}
Crossover, or purification, refers to the process that moves a feasible or optimal solution to a vertex solution, also known as the extreme point, basic feasible solution, or corner solution. Purification methods \citep{charnes1963opposite, charnes1965extreme} were proposed even earlier than those non-vertex LP methods, including the ellipsoid methods \citep{khachiyan1979polynomial} and IPMs \citep{karmarkar1984new}.

Given a basis $B$ and a feasible solution $\xbf$, the purification algorithm \citep{charnes1963opposite, charnes1965extreme} generates a null space vector step $\alphabf$ same as the direction in simplex methods and then tries to move $\xbf$ along it without worsening the objective or update $B$ for feasible directions, which will finally return a feasible vertex with the objective value no worse in finite steps. Given a basis $B$ and a strictly interior feasible point, \citet{kortanek1988new} improved Charnes's method by drawing the vertex back to the interior if it is not optimal. An optimal extreme point will be reached in a finite number of steps. If the feasible point is not entirely in the interior, \citet{kortanek1988new} also provided another purification method that considers dual feasibility. Combining Kortanek's second method with a proper cleanup will return an optimal basis. Given a dual optimal solution $\pibf$, \citet{amor2006recovering} constructed an auxiliary dual LP by adding box constraints around $\pibf$ to the original dual problem. Dualizing back to obtain an auxiliary primal LP, Amor claimed that the auxiliary primal LP can be efficiently solved by simplex methods, and then an optimal vertex can be easily recovered.

Starting from an optimal primal-dual solution pair, \citet{megiddo1991finding} designed a strongly polynomial crossover approach to an optimal vertex. Megiddo also pointed out that given only the optimal primal or dual solution, solving a general LP is theoretically as hard as solving it from scratch. Megiddo's crossover method lays down the algorithmic and theoretical foundation for modern crossover algorithms in modern LP solvers. Our randomized crossover modifies Megiddo's crossover framework, which will be introduced in Section~\ref{sec:cross}.

With the boom of various IPMs and MIP applications, crossover gradually plays a significant role in the LP solvers. The key task becomes refining IPM solutions to optimal vertices. \citet{mehrotra1991finding} added a controlled random cost perturbation to eliminate dual degeneracy at the sacrifice of losing only a little optimality. Especially for non-degenerate LPs, an indicator \citep{tapia1991optimal} can be calculated from the interior solution to identify the basis. \citet{bixby1994recovering} implemented a modified simplex method, which accepts an interior feasible point as a super-basic solution. Then similar to a normal simplex method, Bixby's algorithm uses the primal (or dual) ratio test to push primal (or dual) variables to bound if possible or pivots it into (or out of) the bound. \citet{andersen1996combining} used the strictly complementary property of IPM solutions to define a perturbed LP problem. Andersen's method obtains a basis and then recovers feasibility by pivoting and taking a weighted average with the given IPM solution. \citet{ge2025interior} also proposed a perturbation crossover method that combines variable fixing and random cost perturbation. They proved that as long as the perturbation is sufficiently small, their crossover approach will not bring objective loss.

Our crossover modifies Megiddo's framework and adopts random perturbation, together with PDLP and ordinary least squares (OLS) solver as subroutines. More details will be discussed in Section~\ref{sec:cross}.

\section{Spiral Behavior of PDHG for LP} \label{sec:pdlp}
In this section, we present a geometric viewpoint on the spiral behavior of PDHG on LPs. For ease of presentation, we consider PDHG based on the standard form of LP, while most of the results can be naturally extended to other formulations. More formally, we consider
\begin{equation}
    \begin{aligned}
        \min\  & \cbf^\top \xbf \\
        \st\   & \Abf\xbf = \bbf \\
               & \xbf \geq \zerobf,
    \end{aligned}
    \label{prob:lp}
\end{equation}
where $\Abf\in\Rbb^{m\times n}, \cbf\in\Rbb^n, \bbf\in\Rbb^m$. The dual problem associated with \eqref{prob:lp} is
\begin{equation}
    \begin{aligned}
        \max\  & \bbf^\top \ybf \\
        \st\   & \Abf^\top \ybf \leq \cbf.
    \end{aligned}
    \label{prob:d-lp}
\end{equation}

\subsection{Preliminaries}
\subsubsection{PDHG for LP}

PDHG solves the following saddle-point formulation of \eqref{prob:lp} and \eqref{prob:d-lp}
\begin{equation*}
    \min_{\xbf\geq\zerobf}\max_{\ybf}\ \Lcal(\xbf,\ybf) = \cbf^\top \xbf - \ybf^\top \Abf\xbf + \bbf^\top \ybf,
\end{equation*}
and the detail is presented in Algorithm \ref{alg:pdhg-lp}. Without loss of generality, we assume the primal and the dual step sizes are the same throughout the paper (otherwise, we can rescale the primal or dual variables correspondingly, see \citet{applegate2023faster}).

\begin{algorithm}[!ht]
    \caption{Primal-Dual Hybrid Gradient for LP} \label{alg:pdhg-lp}
    \KwData{The standard LP problem with $\cbf, \Abf, \bbf$.}
    \KwIn{Initial point $(\xbf^{(0)}, \ybf^{(0)}) \in \Rbb^n_+ \times \Rbb^m$, step size $\eta \in (0, \frac{1}{\|\Abf\|_2})$, tolerance $\epsilon \in (0, +\infty)$.}
    \KwOut{The $\epsilon$-optimal solution $(\xbf^{(k)}, \ybf^{(k)})$.}
    $k\leftarrow0$ \;
    \While{not converge within $\epsilon$ accuracy}{
        $\xbf^{(k+1)} \leftarrow \proj_{\Rbb^n_+} \left(\xbf^{(k)} - \eta(\cbf - \Abf^\top \ybf^{(k)}) \right)$ \;
        $\ybf^{(k+1)} \leftarrow \ybf^{(k)} + \eta (\bbf - \Abf(2\xbf^{(k+1)}-\xbf^{(k)}))$ \;
        $k \leftarrow k+1$ \;
    }
\end{algorithm}
 
PDHG can be viewed as conducting projected primal gradient descent and projected dual gradient ascent with extrapolation in turn, and for primal-dual feasible LPs, it is guaranteed to converge to an optimal solution (see, for example, \citet{lu2022infimal,lu2023unified} for intuitions of this extrapolation step). Overall, updates in PDHG are simple, with only matrix-vector multiplication as the bottleneck, while converging to an optimal primal-dual pair at a linear rate.

PDHG, together with other FOMs, has been observed to spiral in practice. For example, \citet{applegate2023faster} plotted a spiral trajectory of PDHG for a 2-dimensional bilinear saddle-point problem. \citet{deng2024enhanced} visualized the first 3 dimensions of ABIP iterations when solving an LP from NETLIB \citep{gay1985electronic}, which also forms an oscillation. These observations encourage FOMs to restart from the average among previous steps, fast approaching the spiral center. From the theoretical aspect, restart accelerates PDHG to converge at a faster linear rate~\citep{applegate2023faster}.

\subsubsection{Infimal Displacement Vector}
For infeasible LPs, PDHG iterates do not converge but rather diverge along a vector called the infimal displacement vector (IDV) of the PDHG operator~\citep{applegate2024infeasibility}. Formally, we denote $\Tbf$ as the PDHG operator, i.e.,
\begin{equation*}
    \Tbf(\xbf^{(k)}, \ybf^{(k)})
    =
    \begin{bmatrix}
        \proj_{\Rbb^n_+} \left(\xbf^{(k)} - \eta(\cbf - \Abf^\top \ybf^{(k)}) \right) \\
        \ybf^{(k)} + \eta (\bbf - \Abf(2\proj_{\Rbb^n_+} \left(\xbf^{(k)} - \eta(\cbf - \Abf^\top \ybf^{(k)}) \right)-\xbf^{(k)}))
    \end{bmatrix}.
\end{equation*}

Its IDV is defined as the minimum perturbation we should subtract from $\Tbf$ to ensure it has a fixed point:
\begin{definition}[Infimal displacement vector \citep{pazy1971asymptotic}] \label{def:idv}
    The unique infimal displacement vector for PDHG with the step size $\eta \in (0, \frac{1}{\|\Abf\|_2})$ is defined as
    \begin{equation*}
        \begin{aligned}
            \vbf_\Tbf = \argmin \  & \frac{1}{2} \|\vbf\|^2_\Mbf                      \\
            \st                 \  & \vbf \in \text{cl}(\text{range}(\Tbf - \Ibf))
        \end{aligned}
    \end{equation*}
    where $\Tbf$ is the operator corresponding to PDHG.
\end{definition}

If $\eta \in (0, \frac{1}{\|\Abf\|_2})$, $\Tbf$ is firmly non-expansive with respect to the norm $\|\cdot\|_\Mbf$~\citep{applegate2024infeasibility}, where the positive definite matrix $\Mbf = 
\begin{bmatrix}
    \Ibf_n & \eta \Abf^\top \\
    \eta \Abf & \Ibf_m
\end{bmatrix}$.
For firmly non-expansive operators, one can recover the IDV using the difference of iterates and/or scaled iterates:

\begin{prop}[\citep{pazy1971asymptotic,baillon1978asymptotic}] \label{prop:idv}
    Let $\Tbf$ be a non-expansive operator and $\zbf$ be the initial point. Then the infimal displacement vector of the operator $\Tbf$ satisfies
    \begin{equation*}
        \lim_{k\to\infty} \frac{\Tbf^k\zbf}{k} = \vbf.
    \end{equation*}
    If further $\Tbf$ is firmly non-expansive, then
    \begin{equation*}
        \lim_{k\to\infty} \Tbf^{k+1}\zbf - \Tbf^k\zbf = \vbf.
    \end{equation*}
\end{prop}

Proposition \ref{prop:idv} implies that PDHG will diverge along $\vbf \neq \zerobf$ if the LP is primal or dual infeasible, and will converge ($\vbf = \zerobf$) if the LP is primal-dual feasible.

\subsection{PDHG Spiral Behavior}
In this part, we formally present the PDHG spiral behavior when solving LPs. At a high level, PDHG identifies different bases from phase to phase. During each phase, the iterates follow a spiral ray, where the ray direction and the spiral center have closed-form formulas.

First, we define the basis of PDHG iterates, which characterizes the active primal variables. This definition is very similar to the basis defined for the simplex method. The fundamental difference is that the corresponding columns of the basis in the constraint matrix may not form a nonsingular square matrix.

\begin{definition}[PDHG basis] \label{def:pdlp-basis}
    For a PDHG iterate $(\xbf,\ybf)$, the basis $B$ and the non-basis $N$ form a partition of the primal variables, defined as
    \begin{equation*}
        \begin{aligned}
            B & = \{i: \xbf_i > 0\} \cup \{i: \xbf_i = 0, \cbf_i - \Abf_i^\top \ybf \leq 0\} \\
            N & = \{i: \xbf_i = 0, \cbf_i - \Abf_i^\top \ybf > 0\}.
        \end{aligned}
    \end{equation*}
    We further call $\xbf_i$ a basic variable if $i\in B$ and $\xbf_i$ a non-basic variable if $i\in N$.
\end{definition}

Then according to changes in the PDHG basis, we can divide the PDHG solving process into multiple phases.
\begin{definition}[Phase and basis change event] \label{def:phase-basis-change}
    We call a sequence of consecutive PDHG iterations form a phase if they share the same PDHG basis. A basis change event occurs when two sequential PDHG iterations have different PDHG bases.
\end{definition}

Within one phase, the spiral ray---including the spiral center and the ray direction, which are key to describing the PDHG spiral behavior---is defined as follows.

\begin{definition}[Spiral ray] \label{def:spiral-ray-center}
    A spiral behavior along the spiral ray is a trajectory $\{\zbf^{(k)}\}$ in the form of
    \begin{equation*}
        \zbf^{(k)} - \zbf_\vbf = \Pbf^k (\zbf^{(0)} - \zbf_\vbf) + k\vbf,
    \end{equation*}
    where the eigenvalues of the diagonalizable $\Pbf$ consist of contingent one and non-real complex eigenvalues with the modulus strictly less than one, $\zbf^{(0)} - \zbf_\vbf$ is orthogonal to all eigenvectors associated with the eigenvalue of one, and $\vbf^\top \Pbf^k (\zbf^{(0)} - \zbf_\vbf) = \zerobf$ for all $k$. We refer to the vector $\vbf$ as the ray direction, the vector $\zbf_\vbf$ as the spiral center, and together $\zbf_\vbf + \theta \vbf, \theta \geq 0$ as the spiral ray.
\end{definition}

The $\Pbf^k (\zbf^{(0)} - \zbf_\vbf)$ captures the rotation with the radius decreasing linearly along the direction corresponding to the complex eigenvalues, while the $k\vbf$ captures the forward movement. The orthogonality between rotation and forward movement allows us to analyze their effects independently.

To gain intuition about the above definition of PDHG basis, phase, basis change event, and spiral ray, we illustrate them with the following example:
\begin{example}[PDHG spiral and basis change events] \label{exp:bases}
    We apply PDHG to solve a toy LP with two primal variables and one constraint.
    \begin{equation}\label{eq:toy}
        \begin{aligned}
            \min \ & 2\xbf_1 + 3 \xbf_2 \\
            \st  \ & \xbf_1 + 2\xbf_2 = 1 \\
                   & \xbf_1, \xbf_2 \geq 0.
        \end{aligned}
    \end{equation}
    PDHG uses the step size $\eta = 0.05$ and relative tolerance $\epsilon=1e-8$, and starts from $\xbf^{(0)} = (1, 2), \ybf^{(0)} = 2$. It takes 3235 steps to the unique optimal solution $\xbf^* = (0, 0.5), \ybf^* = 1.5$.
\end{example}

\begin{figure}[!ht]
    \centering
    \subcaptionbox{Four phases of PDHG on the toy LP \eqref{eq:toy}. \label{fig:pdlp-spiral-demo}}{
        \includegraphics[width=0.5\linewidth]{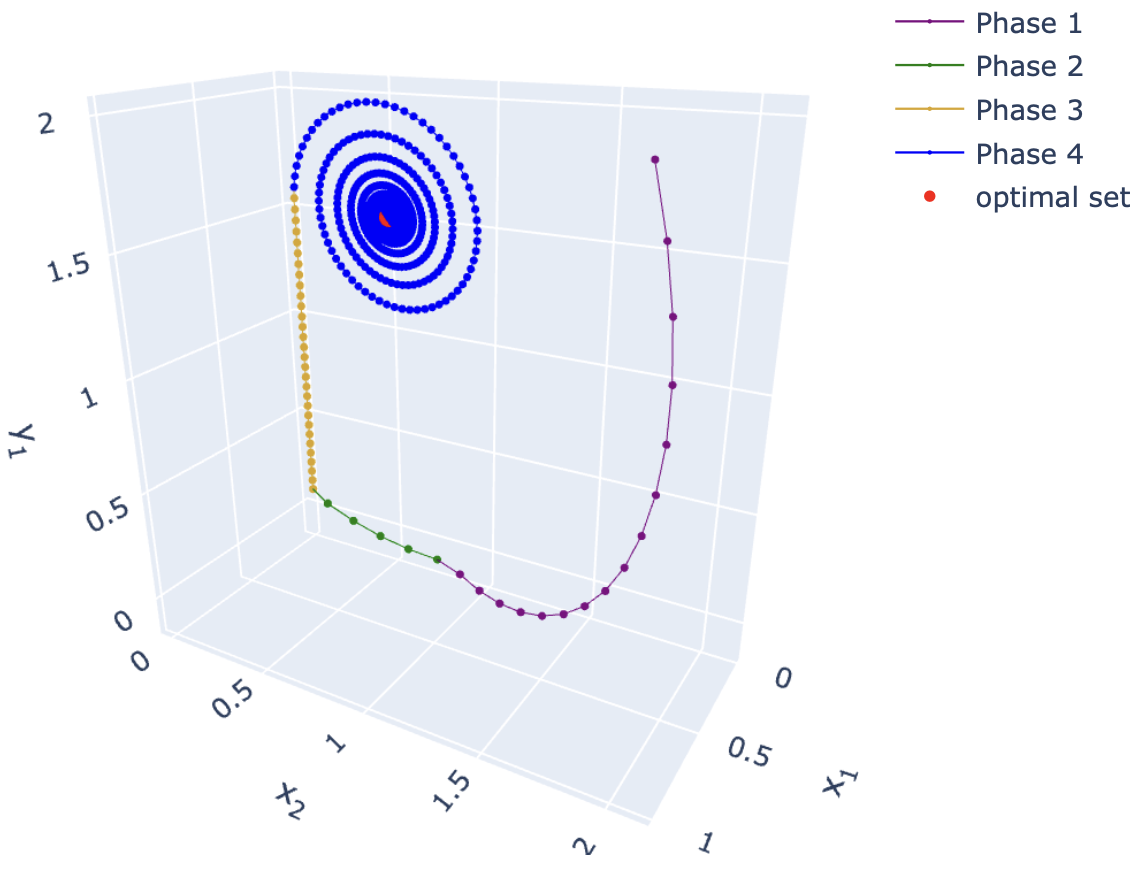}
    }
    \hspace{1em}
    \subcaptionbox{The trajectory of PDHG for $\min_\xbf 2\xbf_1+3\xbf_2, \st, \xbf_1+2\xbf_2=1$. \label{fig:phase-1-nobd}}{
        \includegraphics[width=0.4\linewidth]{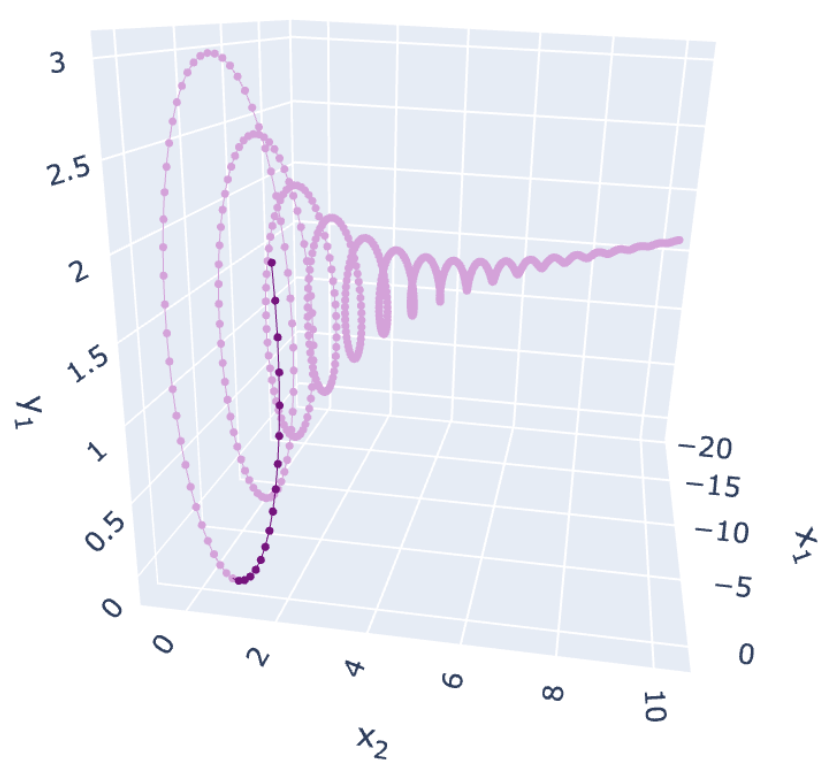}
    }
    \caption{The spiral behavior of PDHG.}
\end{figure}

The whole trajectory of the primal-dual solution $(\xbf, \ybf)$ is shown in Figure \ref{fig:pdlp-spiral-demo}, which can be divided into four phases.
\begin{itemize}
    \item Phase 1 (in purple)
    \begin{itemize}
        \item The first 16 steps form a huge arc while moving forward. Then PDHG hits the $\xbf_1 = 0$ plane and a basis change event happens.
        \item During Phase 1, we have $B=\{1,2\}$ and $N=\emptyset$.
    \end{itemize}
    \item Phase 2 (in green)
    \begin{itemize}
        \item During iterations from steps 17 to 21, PDHG arcs around within the $\xbf_1 = 0$ plane until it reaches another $\xbf_2 = 0$ plane.
        \item During Phase 2, we have $B=\{2\}$ and $N=\{1\}$.
    \end{itemize}
    \item Phase 3 (in khaki)
    \begin{itemize}
        \item Marching against $\xbf_1 = 0$ and $\xbf_2 = 0$, the trajectory degenerates to a ray along $\ybf_1$ axis.
        \item During Phase 3, we have $B=\emptyset$ and $N=\{1,2\}$.
    \end{itemize}
    \item Phase 4 (in blue)
    \begin{itemize}
        \item Finally, $\xbf_2$ leaves its bound and the optimal PDHG basis is identified. There is no more basis change event and PDHG spirals towards the optimal solution.
        \item During Phase 4, we have $B=\{2\}$ and $N=\{1\}$.
    \end{itemize}
\end{itemize}

Note that PDHG has already found the optimal PDHG basis in Phase 2, but changes the basis two more times due to the obstruction from the lower bound of $\xbf_2$. Actually, Phase 2 and Phase 4 share the same spiral center, and the directions of their spiral rays are both zero vectors.

Another interesting fact is that no basis change event will happen if we remove the bounds $\xbf_1,\xbf_2 \geq 0$. Phase 1 will continue and be fully observed as a combination of rotation and forward movement, as displayed in Figure \ref{fig:phase-1-nobd}. Moreover, the rotation converges rapidly; thus, PDHG may diverge almost straight along the IDV, which is an instance of the spiral ray.

The basis change events split the PDHG trajectory into multiple similar spiral phases. Our next goal is to analyze the spiral behavior within each phase, when the PDHG basis remains unchanged.

\subsubsection{PDHG Spiral Behavior within One Phase}
Within one phase, the basis of the PDHG iterates does not change, and it turns out the PDHG iterates follow a spiral ray, as observed in the example stated in the previous section. This can be formalized in the next theorem:
\begin{theorem}[The spiral behavior of PDHG] \label{thm:pdlp-spiral}
    Within one phase given the PDHG basis $(B, N)$, PDHG iterates as
    \begin{equation} \label{eq:iter}
        \begin{bmatrix}
            \xbf_B^{(k)} - (\xbf_\vbf)_B \\ \ybf^{(k)} - \ybf_\vbf
        \end{bmatrix}
        =
        \Pbf_B^k
        \begin{bmatrix}
            \xbf_B^{(0)} - (\xbf_\vbf)_B \\ \ybf^{(0)} - \ybf_\vbf
        \end{bmatrix}
        +
        k
        \begin{bmatrix}
            (\vbf_\xbf)_B \\ \vbf_\ybf
        \end{bmatrix},
        \ \xbf_N^{(k)} = \zerobf,
    \end{equation}
    where $\zbf^{(0)} = (\xbf^{(0)}, \ybf^{(0)})$ is the initial point of the phase, the matrix
    \begin{equation} \label{eq:sp-operator}
        \Pbf_B = \begin{bmatrix}
            \Ibf_{|B|}    & \eta \Abf_B^\top                   \\
            - \eta \Abf_B & \Ibf_m - 2\eta^2 \Abf_B\Abf_B^\top
        \end{bmatrix},
    \end{equation}
    the spiral center $\zbf_\vbf = (\xbf_\vbf, \ybf_\vbf)$ is given by
    \begin{equation} \label{eq:sp-center}
        \begin{aligned}
            (\xbf_\vbf)_B &= (\Abf_B^\top \Abf_B)^\dagger \Abf_B^\top \bbf + \proj_{\Abf_B\xbf = \zerobf}(\xbf_B^{(0)}),\ (\xbf_\vbf)_N = \zerobf \\
            \ybf_\vbf &= (\Abf_B \Abf_B^\top)^\dagger \Abf_B \cbf_B + \proj_{\Abf_B^\top\ybf = \zerobf}(\ybf^{(0)}),
        \end{aligned}
    \end{equation}
    and the ray direction $\vbf = (\vbf_\xbf, \vbf_\ybf)$ is given by
    \begin{equation} \label{eq:sp-ray}
        \begin{aligned}
            (\vbf_\xbf)_B &= - \eta \left[\cbf_B - \Abf_B^\top(\Abf_B \Abf_B^\top)^\dagger \Abf_B \cbf_B\right],\ (\vbf_\xbf)_N = \zerobf \\
            \vbf_\ybf &= \eta \left[\bbf - \Abf_B(\Abf_B^\top \Abf_B)^\dagger \Abf_B^\top \bbf\right].
        \end{aligned}
    \end{equation}
    Moreover, the rotation part converges to $\zerobf$, i.e.,
    \begin{equation} \label{eq:sp-cond-lim}
        \lim_{k\to\infty}
        \Pbf_B^k
        \begin{bmatrix}
            \xbf_B^{(0)} - (\xbf_\vbf)_B \\ \ybf^{(0)} - \ybf_\vbf
        \end{bmatrix}
        =
        \zerobf,
    \end{equation}
    and the forward movement is orthogonal to the rotation, i.e.,
    \begin{equation} \label{eq:sp-cond-orth}
        \vbf_B^\top \Pbf_B^k
        \begin{bmatrix}
            \xbf_B^{(0)} - (\xbf_\vbf)_B \\ \ybf^{(0)} - \ybf_\vbf
        \end{bmatrix}
        = 0, \ \forall k,
    \end{equation}
    where $\vbf_B = ((\vbf_\xbf)_B, \vbf_\ybf)$.
\end{theorem}
The proof is given in Appendix \ref{apdx:prove-thm-pdlp-spiral}.

Equation \eqref{eq:iter} presents a closed-form solution of the iterates within one phase. Intuitively, PDHG spirals by rotating around $\zbf_\vbf$ and moving forward along $\vbf$. The rotation and forward movement are orthogonal to each other. No matter how PDHG rotates in the hyperplane $\vbf^\top\zbf = 0$, it will firmly step one unit of $\vbf$ along the spiral ray at each iteration. The ray direction is unique for one phase regardless of the initial solution $\zbf^{(0)}$, but the spiral centers may depend on the initial solution (i.e., the projection term in \eqref{eq:sp-center}).

To be more specific, within one phase, PDHG classifies $\xbf$ into $\xbf_B$ and $\xbf_N$, where $\xbf_B$ are movable and $\xbf_N$ are temporarily treated as if fixed at their bounds. $\Pbf_B$ is the linear operator that causes rotation. $(\xbf_\vbf)_B$ is part of the spiral center corresponding to the movable part of $\xbf$. $(\vbf_\xbf)_B$ is part of the ray direction corresponding to the movable part of $\xbf$.

Furthermore, in terms of the influence of the step size $\eta$, in general, a larger step size will not only help the rotation converge faster by reducing the moduli of the eigenvalues of $\Pbf_B$ but also accelerate the forward movement along the ray direction.

\subsubsection{Understanding How the Spiral Behavior Helps the Convergence}

Theorem \ref{thm:pdlp-spiral} presents orthogonally decomposed spiral behavior of PDHG into forward movement along the ray direction $\vbf$ and rotation within the hyperplane $\vbf^\top\zbf = 0$. These evidences almost form a satisfying description of what PDHG is doing in one phase. Here we would like to further discuss how the trajectory of a spiral can lead PDHG to the optimum.

The optimality of an LP contains primal feasibility, dual feasibility, and duality gap. Considering how the PDHG spiral affects all three aspects, we have the following proposition:

\begin{prop} \label{prop:opt-improve}
    Within one phase given the PDHG basis $(B, N)$, the rotation improves the primal and dual feasibility by approaching the spiral center with the least squares primal error $\|\Abf_B \xbf_B - \bbf\|_2$ and dual error $\|\Abf_B^\top \ybf - \cbf_B\|_2$.
    The forward movement improves the duality gap at the spiral center in the sense that the duality gap along the spiral ray monotonically decays, i.e., $\cbf^\top \vbf_\xbf - \bbf^\top \vbf_\ybf \leq 0$. Furthermore, if the forward movement is nonzero (i.e., $\vbf\not=\zerobf$), we have a strict decay of the duality gap when moving along the direction of $\vbf$, i.e., $\cbf^\top \vbf_\xbf - \bbf^\top \vbf_\ybf < 0$.
\end{prop}
The proof is provided in Appendix \ref{apdx:prove-prop-opt-improve}.

The primal feasibility $\Abf \xbf = \bbf, \xbf \geq \zerobf$ is naturally improved via rotating to minimize the difference between $\Abf\xbf$ and $\bbf$. In contrast, when optimizing the dual feasibility, PDHG not only aims to satisfy $\Abf^\top \ybf \leq \cbf$, but also prefers to activate $\Abf_B^\top \ybf = \cbf_B$. This can be explained from the perspective of complementary slackness. In a primal-dual optimal solution $(\xbf, \ybf)$ to \eqref{prob:lp}, the complementary slackness condition requires $\xbf_i (\cbf_i - \Abf_i^\top \ybf) = 0, \forall i$. Therefore, for those primal variables in $B$, it is better to activate the corresponding dual constraints for optimality.

As for the effect of rotation in the duality gap, damped oscillation occurs around the primal and dual objectives at the spiral center, resulting in the total duality gap exhibiting an oscillating improvement.

\subsubsection{Basis Change Event}
Basis change events happen between two adjacent periods when the spiral hits new bounds or escapes from current bounds, i.e., for the phase with PDHG basis $(B, N)$, the next step $(\xbf^+, \ybf^+)$ satisfies
\begin{equation} \label{eq:leave-basis}
    \xbf_i^+ = 0, \ \cbf_i - \Abf_i^\top\ybf^+ > 0, \ \exists i \in B,
\end{equation}
or
\begin{equation} \label{eq:enter-basis}
    \cbf_i - \Abf_i^\top \ybf^+ \leq 0, \ \exists i \in N.
\end{equation}
We call the condition \eqref{eq:leave-basis} the ``leaving basis'' condition, since the basic coordinate $i\in B$ satisfying \eqref{eq:leave-basis} leaves the basis in the next iteration. Similarly, we call condition \eqref{eq:enter-basis} the ``entering basis'' condition, since the non-basic coordinate $i\in N$ satisfying \eqref{eq:enter-basis} enters the basis in the next iteration.

To be more specific, the basis change event corresponds to the activation of $\xbf_B \geq \zerobf$ and the satisfaction of $\Abf_N^\top\ybf < \cbf_N$. From the perspective of the spiral, as shown in Figure \ref{fig:basis-change}, three typical reasons will cause a basic variable to leave the basis:
\begin{itemize}
    \item First, the spiral center is located outside the bounds (and forward movement is ignored for simplicity). Since rotation converges to the center, PDHG will meet bounds before converging, yielding a basis change event.
    \item Second, although the spiral center is located inside the bounds (and forward movement is ignored for simplicity), the spiral radius, i.e., the distance between the current point and the spiral center, is too large.
    \item Third, the spiral ray has a strictly negative component in its direction, which will eventually hit the bound.
\end{itemize}
For the entering basis event, when $\ybf^+$ violates the dual constraints $\Abf_N^\top\ybf < \cbf_N$, the corresponding components of $\xbf^+_N$ have a tendency to increase, making the projection operator redundant in the next step.

It is worth highlighting that, differing from the optimality improvement in one phase, the basis change event is usually triggered by the combined action of rotation and forward movement, instead of by one side alone. For example, though the spiral center may lie outside the bounds as shown in Figure \ref{fig:rotate1-bc}, the spiral ray is likely to intersect the bounds, leading to a basis change event similar to the second case in Figure \ref{fig:rotate2-bc}.
\begin{figure}[!ht]
    \centering
    \subcaptionbox{The spiral center is outside the bounds. \label{fig:rotate1-bc}}{
        \includegraphics[height=0.2\linewidth]{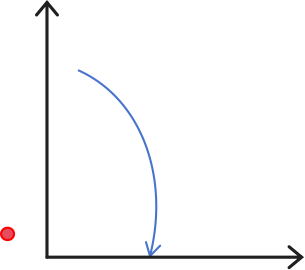}
    }
    \hfill
    \subcaptionbox{The spiral center is inside the bounds, but the radius is too large. \label{fig:rotate2-bc}}{
        \includegraphics[height=0.2\linewidth]{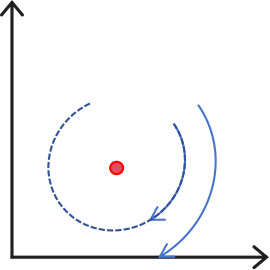}
    }
    \hfill
    \subcaptionbox{The ray hits the bounds. \label{fig:line-bc}}{
        \includegraphics[height=0.2\linewidth]{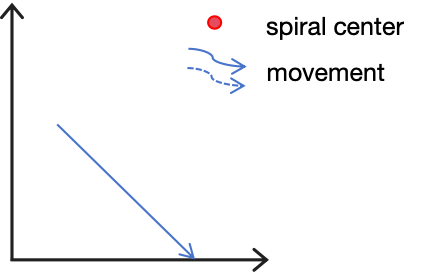}
    }
    \caption{Three typical cases for the basis change in the first quadrant.} \label{fig:basis-change}
\end{figure}

\section{Crossover Inspired by PDHG} \label{sec:cross}

Without loss of generality, we assume $\Abf$ is full row rank in this section. For a primal-dual feasible LP, PDLP returns only a primal-dual optimal solution pair with PDHG basis $(B, N)$ satisfying \eqref{eq:pdhg-optimal}
\begin{equation} \label{eq:pdhg-optimal}
    \begin{aligned}
        & \Abf_B \xbf_B = \bbf, \ \xbf_B \geq \zerobf,\ \xbf_N = \zerobf        \\
        & \cbf_B - \Abf_B^\top \ybf = \zerobf, \ \cbf_N - \Abf_N^\top \ybf > \zerobf,
    \end{aligned}
\end{equation}
rather than the optimal vertex with basis $(B, N)$ satisfying \eqref{eq:vertex-optimal}
\begin{equation} \label{eq:vertex-optimal}
    \begin{aligned}
        \xbf_B &= (\Abf_B)^{-1} \bbf \geq \zerobf, \ \xbf_N = \zerobf \\
        \ybf   &= (\Abf_B^\top)^{-1} \cbf_B, \ \cbf_N - \Abf_N^\top \ybf \geq \zerobf.
    \end{aligned}
\end{equation}
Compared with \eqref{eq:pdhg-optimal}, here a vertex must be determined by $|B| = m$ linearly independent columns of $\Abf$, and there may also be zero components in $\cbf_N - \Abf_N^\top \ybf$. Crossover for PDLP refers to obtaining a solution that satisfies \eqref{eq:vertex-optimal} from a solution satisfying \eqref{eq:pdhg-optimal}.

One important observation is that if we
properly fix variables on bounds, then a feasible solution to LP is also an optimal solution to the LP, which is formally stated below in Proposition \ref{prop:fix-lp}.
\begin{prop} \label{prop:fix-lp}
    Given the primal-dual optimal solution $(\xbf, \ybf)$ to \eqref{prob:lp}, denote $B = \{i:\xbf_i > 0\}$, $E = \{i: \xbf_i = 0\}$, $D = \{i:\cbf_i - \Abf_i^\top \ybf = 0\}$, and $N = \{i: \cbf_i - \Abf_i^\top \ybf > 0\}$. Any feasible solution $(\tilde{\xbf}, \tilde{\ybf})$ satisfying
    \begin{equation} \label{eq:fix-lp}
        \begin{aligned}
            & \Abf_B \tilde{\xbf}_B = \bbf, \ \tilde{\xbf}_B \geq \zerobf, \ \tilde{\xbf}_E = \zerobf \\
            & \cbf_D - \Abf_D^\top \tilde{\ybf} = \zerobf, \ \cbf_N - \Abf_N^\top \tilde{\ybf} \geq \zerobf
        \end{aligned}
    \end{equation}
    also form a primal-dual optimal solution to the original LP.
\end{prop}
The proof is given in Appendix \ref{apdx:prove-prop-fix-lp}.

Therefore, we can iteratively push $\tilde{\xbf}_B$ and $\Abf_N^\top\tilde{\ybf}$ in \eqref{eq:fix-lp} to their bounds and fix those on bounds until a vertex is found, which is precisely the idea of \citet{megiddo1991finding}.

Another important phenomenon, which is the key to our crossover method, is that when PDHG starts from a primal (or dual) feasible solution within one phase, it also becomes the start of the primal (or dual) PDHG spiral ray.
\begin{prop} \label{prop:feas-point}
    Within one phase with PDHG basis $(B, N)$ and the initial solution $(\xbf^{(0)}, \ybf^{(0)})$, if $\Abf_B\xbf^{(0)}_B = \bbf$, then we have $\xbf_\vbf = \xbf^{(0)}$. Similarly, if $\Abf_B^\top \ybf^{(0)} = \cbf_B$, then we have $\ybf_\vbf = \ybf^{(0)}$.
\end{prop}
The proof is provided in Appendix \ref{apdx:prove-prop-feas-point}.

During each crossover push, Proposition \ref{prop:feas-point} inspires us to guide those unfixed variables along the PDHG ray direction.

\subsection{Crossover Framework}

Many crossover frameworks have been designed \citep{megiddo1991finding, bixby1994recovering, andersen1996combining}, most combining simplex methods and IPMs. Following Megiddo's idea, our crossover mainly includes three parts:
\begin{enumerate}
    \item Primal push moves $\xbf$ towards their bounds while maintaining the primal feasibility. After this process, the primal part of an optimal vertex solution is found, although degeneracy may exist.
    \item Dual push moves $\ybf$ to activate as many dual constraints as possible while maintaining the dual feasibility.  After this process, the dual part of an optimal vertex solution is obtained.
    \item Linear independence check will no more change solution values but select basic columns to construct the nonsingular square matrix $\Abf_B$.
\end{enumerate}

\begin{algorithm}[!ht]
    \caption{Crossover Framework} \label{alg:crossover-frame}
    \KwData{The standard LP problem with $\cbf, \Abf, \bbf$.}
    \KwIn{Initial optimal solution $(\xbf, \ybf) \in \Rbb^n_+ \times \Rbb^m$.}
    \KwOut{The optimal vertex solution $(\xbf, \ybf)$ and the optimal basis $(B, N)$.}
    
    \tcc{----------------------------------------------------------------------------------------------------}
    \tcc{Primal Push}
    \tcc{----------------------------------------------------------------------------------------------------}
    $B \leftarrow \{i: \xbf_i > 0\}$ \tcp*{initialize basic set}
    \While{primal support can be further reduced}{
        Generate a primal direction $(\deltabf_\xbf)_B \neq \zerobf$ in the kernel of $\Abf_B$ by solving \eqref{prob:ols-dual} and calculating the primal spiral ray in \eqref{eq:sp-ray} with cost perturbation and $\eta = 1$ \;
        \If{$(\deltabf_\xbf)_B \neq \zerobf$}{
            \If{$(\deltabf_\xbf)_B \geq \zerobf$}{
                $(\deltabf_\xbf)_B \leftarrow -(\deltabf_\xbf)_B$ \tcp*{make sure to push basic variables to bounds}
            }
            $\theta \leftarrow \min_{\{i\in B: (\deltabf_\xbf)_i < 0\}}\{-\frac{\xbf_i}{(\deltabf_\xbf)_i}\}$ \tcp*{primal ratio test}
            $\xbf_B \leftarrow \xbf_B + \theta (\deltabf_\xbf)_B$ \tcp*{move $\xbf_B$ to reduce its support}
            $B \leftarrow B \setminus \{i\in B: \xbf_i = 0\}$ \tcp*{update basic set}
        }
    }
    
    \tcc{----------------------------------------------------------------------------------------------------}
    \tcc{Dual Push}
    \tcc{----------------------------------------------------------------------------------------------------}
    $D \leftarrow \{i: \cbf_i - \Abf_i^\top\ybf = 0\}$, $N \leftarrow \{i: \cbf_i - \Abf_i^\top\ybf > 0\}$ \tcp*{initialize dual activation set}
    \While{dual constraints can be further activated}{
        Generate a dual direction $\deltabf_\ybf \neq \zerobf$ in the kernel of $\Abf_D^\top$ by solving \eqref{prob:ols-primal} and calculating the dual spiral ray in \eqref{eq:sp-ray} with right-hand-side perturbation and $\eta = 1$ \;
        \If{$\Abf_N^\top\deltabf_\ybf \neq \zerobf$}{
            \If{$\Abf_N^\top\deltabf_\ybf \leq \zerobf$}{
                $\deltabf_\ybf \leftarrow -\deltabf_\ybf$ \tcp*{make sure to activate constraints}
            }
            $\theta \leftarrow \min_{\{i\in N: \Abf^\top_i\deltabf_\ybf > 0\}}\{\frac{\cbf_i - \Abf_i^\top\ybf}{\Abf^\top_i\deltabf_\ybf}\}$ \tcp*{dual ratio test}
            $\ybf \leftarrow \ybf + \theta \deltabf_\ybf$ \tcp*{move $\ybf$ to activate dual constraints}
            $D \leftarrow D \cup \{i\in N: \cbf_i - \Abf_i^\top\ybf = 0\}$, $N \leftarrow N \setminus \{i\in N: \cbf_i - \Abf_i^\top\ybf = 0\}$ \tcp*{update dual activation set}
        }
    }
    
    \tcc{----------------------------------------------------------------------------------------------------}
    \tcc{Linear Independence Check}
    \tcc{----------------------------------------------------------------------------------------------------}
    Select $m - |B|$ columns $\Abf_C$ from $\Abf_{D\setminus B}$ via LU decomposition and combine them with columns of $\Abf_B$ to form the entire basis \;
    $B \leftarrow B \cup C$, $N \leftarrow \{1, \cdots, n\} \setminus B$ \tcp*{construct the optimal basic set}
\end{algorithm}

The fundamental difference between our crossover algorithm and Megiddo's scheme is that we use directions $(\deltabf_\xbf)_B$ and $\deltabf_\ybf$ inspired by PDHG rather than the pivot direction in simplex methods to move the variables. Another difference from Megiddo's framework is that we check the linear independence after the dual push is completed to separate the two processes more independently, while Megiddo checks the linear independence immediately once the dual push identifies newly activated dual constraints. In terms of theoretical guarantees, following the same analysis of \citet{megiddo1991finding}, we can show that this PDHG-inspired crossover will succeed with probability (w.p.) 1. 

We formally present the crossover scheme in Algorithm \ref{alg:crossover-frame}, which includes three major components:

\paragraph{Primal push.}

After recognizing the basic set $B$, we solve a series of OLS subproblems \eqref{prob:ols-dual} with various randomly generated costs $\tilde{\cbf}_B$ to calculate the PDHG primal ray directions \eqref{eq:sp-ray}.
\begin{equation} \label{prob:ols-dual}
    \min\ \frac{1}{2} \|\Abf_B^\top \ybf - \tilde{\cbf}_B\|_2^2.
\end{equation}
From Proposition \ref{prop:feas-point}, since the current primal solution $\xbf_B$ is still feasible for \eqref{prob:lp} with cost perturbation, it is located at the primal spiral center of the rotation in the current phase of PDHG. We can then start from $\xbf_B$ and move along $(\deltabf_\xbf)_B$ until reaching new boundaries, mirroring how the PDHG solves the cost-perturbed LPs. At the end of the primal push, no more primal variable can be pushed to its bound when columns of $\Abf_B$ are linearly independent. OLS can be solved by the direct QR factorization, LSMR \citep{fong2011lsmr}, or the conjugate gradient (CG) method.

\paragraph{Dual push.}

In the dual push, similarly, we compute a series of OLS subproblems \eqref{prob:ols-primal} with perturbed right-hand side $\tilde{\bbf}$ to calculate the dual ray directions of PDHG~\eqref{eq:sp-ray}.
\begin{equation} \label{prob:ols-primal}
    \min\ \frac{1}{2} \|\Abf_D \xbf_D - \tilde{\bbf}\|_2^2 \ .
\end{equation}
More constraints are fixed to be active step by step until rows of $\Abf_D$ are linearly independent when $\ybf$ will be uniquely determined by $\Abf_D^\top \ybf = \cbf_D$.

\paragraph{Linear independence check.}

After the primal and dual push, $(\xbf, \ybf)$ should reach some vertex. All we need is to identify $m - |B|$ columns from $\Abf_{C} = \Abf_{D\setminus B}$ to form an optimal basis with $\Abf_B$. Here we apply LU factorization twice.
We first decompose $\Abf_B$ as
\begin{equation*}
    \Abf_B = \Lbf_B \Ubf_B
    =
    \begin{pmatrix}
        \Lbf_{R_1, B} & \phantom{\zerobf} \\
        \Lbf_{R_2, B} & \Ibf_{m-|B|}
    \end{pmatrix}
    \begin{pmatrix}
        \Ubf \\ \zerobf
    \end{pmatrix},
\end{equation*}
where $(R_1, R_2)$ divides the rows according to the decomposition. The basic candidates become
\begin{equation*}
    \begin{aligned}
        \Abf_{B\cup C}
             & =
            \begin{bmatrix}
                \Abf_{R_1, B} & \Abf_{R_1, C} \\
                \Abf_{R_2, B} & \Abf_{R_2, C}
            \end{bmatrix}
            =
            \begin{bmatrix}
                \begin{pmatrix}
                    \Lbf_{R_1, B} & \phantom{\zerobf} \\
                    \Lbf_{R_2, B} & \Ibf
                \end{pmatrix}
                \begin{pmatrix}
                    \Ubf \\ \zerobf
                \end{pmatrix}
                 &
                \begin{matrix}
                    \Abf_{R_1, C} \\ \Abf_{R_2, C}
                \end{matrix}
            \end{bmatrix}       \\
             & =
            \begin{bmatrix}
                \Lbf_{R_1, B} & \phantom{\zerobf} \\
                \Lbf_{R_2, B} & \Ibf
            \end{bmatrix}
            \begin{bmatrix}
                \begin{matrix}
                    \Ubf \\ \zerobf
                \end{matrix}
                 &
                \begin{pmatrix}
                    \Lbf_{R_1, B} & \phantom{\zerobf} \\
                    \Lbf_{R_2, B} & \Ibf
                \end{pmatrix}^{-1}
                \begin{pmatrix}
                    \Abf_{R_1, C} \\ \Abf_{R_2, C}
                \end{pmatrix}
            \end{bmatrix} \\
             & =
            \begin{bmatrix}
                \Lbf_{R_1, B} & \phantom{\zerobf} \\
                \Lbf_{R_2, B} & \Ibf
            \end{bmatrix}
            \begin{bmatrix}
                \begin{matrix}
                    \Ubf \\ \zerobf
                \end{matrix}
                 &
                \begin{pmatrix}
                    \Lbf_{R_1, B}^{-1}            & \phantom{\zerobf} \\
                    -\Lbf_{R_2, B}\Lbf_{R_1, B}^{-1} & \Ibf
                \end{pmatrix}
                \begin{pmatrix}
                    \Abf_{R_1, C} \\ \Abf_{R_2, C}
                \end{pmatrix}
            \end{bmatrix} \\
             & =
            \begin{bmatrix}
                \Lbf_{R_1, B} & \phantom{\zerobf} \\
                \Lbf_{R_2, B} & \Ibf
            \end{bmatrix}
            \begin{bmatrix}
                \Ubf & \Lbf_{R_1, B}^{-1}\Abf_{R_1, C} \\
                \zerobf & \Abf_{R_2, C} - \Lbf_{R_2, B}\Lbf_{R_1, B}^{-1}\Abf_{R_1, C}
            \end{bmatrix}.
    \end{aligned}
\end{equation*}
Then, we only need to select $m - |B|$ linearly independent columns from $\Abf_{R_2, C} - \Lbf_{R_2, B}\Lbf_{R_1, B}^{-1}\Abf_{R_1, C}$, where the second LU decomposition is involved.

\paragraph{Complexity results.}

Our crossover will succeed w.p. 1 without too many OLS subproblems, as stated in Theorem \ref{thm:ols-number}. Moreover, the number of OLS subproblems is strongly polynomial w.p. 1.

\begin{theorem} \label{thm:ols-number}
    Consider the proposed crossover framework (Algorithm \ref{alg:crossover-frame}). With probability 1, the algorithm terminates and outputs a primal-dual optimal vertex solution with at most $n$ primal and/or dual push steps. Consequently, it requires solving at most $n$ OLS subproblems. 
    
\end{theorem}
The proof is provided in Appendix \ref{apdx:prove-thm-ols-number}.

\subsection{Practical Consideration}

We discuss several practical considerations when implementing our proposed crossover scheme.

\paragraph{Auxiliary LP heuristic.}

To reduce the required number of OLS subproblems, inspired by \citet{ge2025interior}, we adopt auxiliary LPs as heuristic in primal and dual push to sparsify the original solution. In primal push, we build an auxiliary LP with random cost perturbation $\tilde{\cbf}_B$,
\begin{equation} \label{prob:aux-lp}
    \begin{aligned}
        \min\  & \tilde{\cbf}_B^\top \xbf_B \\
        \st\   & \Abf_B \xbf_B = \bbf \\
               & \xbf_B \geq \zerobf,
    \end{aligned}
\end{equation}
which can be solved with PDLP or IPMs. This auxiliary LP has multiple optimal solutions w.p. 0 (see Lemma 1 in \citet{ge2025interior}), so with proper perturbation to avoid the unbounded case, we have the unique optimal solution $\xbf_B$ to \eqref{prob:aux-lp}. Then $(\xbf_B, \zerobf_E)$ can be the primal part of an optimal vertex of \eqref{prob:lp}.

For the auxiliary LP in dual push, we sum up the non-negative dual slacks as a penalty to activate more dual constraints.
\begin{equation} \label{prob:aux-lp-dual}
    \begin{aligned}
        \min\  & \onebf^\top (\cbf_N - \Abf_N^\top \ybf) \\
        \st\   & \Abf_D^\top \ybf = \cbf_D \\
               & \Abf_N^\top \ybf \leq \cbf_N.
    \end{aligned}
\end{equation}

Due to the numerical residuals in practice, we may not correctly fix all the variables on bounds at once. To reduce the negative impact of the issue, the amplitude of disturbance should be controlled, see \citet{mehrotra1991finding, ge2025interior} for more specific discussions.

\paragraph{LP formulation.}

So far, we have built our theory and algorithm design on the standard LP \eqref{prob:lp}. In practice, the general LP formulation \eqref{prob:general-lp} is more convenient and thus widely used.
\begin{equation} \label{prob:general-lp}
    \begin{aligned}
        \min\  & \cbf^\top \xbf \\
        \st\   & \lbf_\wbf \leq \Abf \xbf \leq \ubf_\wbf \\
               & \lbf_\xbf \leq      \xbf \leq \ubf_\xbf.
    \end{aligned}
\end{equation}
To extend the concept of the basis on the general form, we reformulate it as \eqref{prob:reform-lp} via introducing the artificial variable $\wbf$.
\begin{equation} \label{prob:reform-lp}
    \begin{aligned}
        \min\  & \cbf^\top \xbf \\
        \st\   & \Abf \xbf - \wbf = - \delta \onebf \\
               & \lbf_\xbf \leq \xbf \leq \ubf_\xbf \\
               & \lbf_\wbf + \delta \onebf \leq \wbf \leq \ubf_\wbf + \delta \onebf,
    \end{aligned}
\end{equation}
where $\delta = 10$ aims to avoid the right-hand-side term becoming $\zerobf$, as the presence of $\zerobf$ may render the LP highly sensitive to perturbations, causing the optimal solution to change drastically. This reformulation guarantees the constraint matrix $\begin{bmatrix} \Abf & -\Ibf\end{bmatrix}$ to be full row rank, where we can easily apply the definition of the basic solution.

\paragraph{Non-basis identification.}

At the beginning of primal and dual push, we need to determine $B$ and $(D, N)$ to construct the auxiliary LP and OLS. For an IPM solution, which tends to stay away from bounds, the central path and strict complementarity are crucial for identifying non-basic variables. Identification is typically done by comparing the orders of magnitude of $\xbf$ and $\cbf - \Abf^\top\ybf$. In contrast, PDHG can benefit from the sparsity of its solution, induced by the projection, to more aggressively fix variables to their bounds.

We take
\begin{equation*}
    B = \{i: \xbf_i > \max(\gamma (\cbf_i - \Abf_i^\top\ybf), \epsilon) \},
\end{equation*}
and after $B$ is updated in primal push
\begin{equation*}
    \begin{aligned}
        D &= \{i: \cbf_i - \Abf_i^\top\ybf \leq \epsilon \} \cup B \\
        N &= \{i: \cbf_i - \Abf_i^\top\ybf > \epsilon \} \setminus B,
    \end{aligned}
\end{equation*}
with $\gamma = 1$ and $\epsilon = 1e-8$. This criterion works well in our experiments for purifying the PDLP solutions.

\paragraph{Perturbation.}

The cost and right-hand-side perturbations guide the direction of crossover. They are perturbed as
\begin{equation*}
    \begin{aligned}
        \tilde{\cbf}_B &= \frac{1}{\|\cbf_B\|_\infty + 1}\cbf_B + \deltabf_\cbf \\
        \tilde{\bbf} &= \frac{1}{\|\bbf\|_\infty + 1}\bbf + \deltabf_\bbf,
    \end{aligned}
\end{equation*}
where $\deltabf_\cbf, \deltabf_\bbf$ are independently and uniformly distributed over the interval $[0, 1]$. The first normalized terms function as a center to stabilize the disturbance, preventing the unfixed non-basic variables from occasionally escaping boundaries.

\paragraph{Efficiency of OLS.}

Note that to obtain the ray directions of PDHG~\eqref{eq:sp-ray}, if the direct QR factorization is applied in the OLS subproblems \eqref{prob:ols-dual} and \eqref{prob:ols-primal}, we need to factorize $\Abf_B^\top$ and $\Abf_D$ respectively. An alternative practice is to utilize the property of Moore–Penrose inverse that $(\Abf_B^\top \Abf_B)^\dagger \Abf_B^\top = \Abf_B^\dagger = \Abf_B^\top (\Abf_B \Abf_B^\top)^\dagger$. We can solve
\begin{equation} \label{prob:ols-dual-alter}
    \min\ \frac{1}{2}\| \Abf_B \dbf_B - \Abf_B\tilde{\cbf}_B \|_2^2,
\end{equation}
and calculate the primal direction $(\deltabf_\xbf)_B$ as
\begin{equation*}
    (\deltabf_\xbf)_B = - \eta \left[\cbf_B - \Abf_B^\top(\Abf_B \Abf_B^\top)^\dagger \Abf_B \tilde{\cbf}_B\right] = - \eta \left(\cbf_B - \dbf_B\right).
\end{equation*}
Consequently, we can focus solely on the factorization of $\Abf_B$.

Currently, all the OLS problems are solved independently from scratch. Note that the coefficient matrices of two adjacent OLS problems differ in only a few columns or rows; thus, further development on reusing the factorization (for example, via Givens rotations) may largely enhance the efficiency of our crossover to the level of the classic simplex-based crossover.

\paragraph{Understanding solution structures between PDLP and IPMs.}

This crossover scheme can be applied to solutions obtained by IPMs or PDLP. However, their structures differ as discussed below.

IPMs find the analytical center of the optimal solution face, and the obtained solutions tend to be denser, while PDLP solutions are usually much sparser because of the projection operator.

IPMs and PDHG both have a strong ability to distinguish non-basic variables. IPMs can identify the strictly complementary solution pair with a sufficiently small duality gap \citep{andersen1996combining}. PDHG can correctly classify variables as basic and non-basic in the scope of PDHG basis within finite iterations \citep{lu2024geometry}.

Therefore, even though the PDLP solution may be less accurate than the IPM solution, it can still be refined to the vertex successfully in practice.

\section{Experiments} \label{sec:exp}

In this section, we present experimental results to demonstrate the competence of our crossover algorithm. Our experiment does not plan to surpass the state-of-the-art crossover code implemented in commercial solvers but rather to verify the effectiveness of our proposed crossover algorithm. Further practical improvements are necessary to explore performance enhancement, which is beyond the scope of this paper.

\paragraph{Benchmark dataset.}
We conduct our crossover experiment on the NETLIB collection \citep{gay1985electronic}. Considering the efficiency of the experiment, we select 100 instances from NETLIB according to two rules:
\begin{itemize}
    \item The number of rows or columns of $\Abf$ in the general form \eqref{prob:general-lp} is smaller than 5000.
    \item cuPDLP-C \citep{lu2023cupdlpc} can solve it to the relative tolerance of $1e-8$ within 600 seconds.
\end{itemize}

\paragraph{Software and hardware.}
Tests are run on a MacBook Pro with an 8-core Apple M2 CPU and 24GB unified memory. Algorithm \ref{alg:crossover-frame} is implemented in Julia and can be accessed at \href{https://github.com/MIT-Lu-Lab/crossover}{https://github.com/MIT-Lu-Lab/crossover}.

\paragraph{Initialization and time limit.}
We first solve the NETLIB LPs using cuPDLP-C inside COPT 7.2 \citep{copt} without GPU acceleration to the relative tolerance of $1e-8$ or $1e-6$. Then the optimal primal-dual solution is passed to the Julia implementation of our crossover algorithm. A crossover time limit of 300 seconds is set and the random seed is set to 0 for reproducibility.

\paragraph{OLS solving.}
OLS solving to identify the spiral ray direction is a critical step in our crossover algorithm. In our experiments, all OLS subproblems are solved by direct QR factorization.

\paragraph{Auxiliary LP solving.}
The auxiliary LPs in our crossover experiment are solved to the relative tolerance of $1e-8$ by cuPDLP-C without GPU acceleration. The heuristic is disabled by default.

\paragraph{Vertex solutions from crossover.}
From Theorem \ref{thm:ols-number}, our crossover will provide us with a primal-dual optimal vertex solution w.p. 1 if the original solution is primal-dual optimal. In practice, due to the numerical residual of the original solution, the vertex from the crossover may sometimes be only primal (or dual) optimal but dual (or primal) infeasible, or even near-optimal but slightly primal-dual infeasible. Since simplex methods can be applied to efficiently eliminate this infeasibility after crossover, such vertices are acceptable.

\paragraph{Results.}

The complete numerical results of the 100 NETLIB instances are shown in Appendix \ref{apdx:netlib-table}. For the original $1e-8$ solutions, our crossover returns 95 primal-dual optimal vertices, 1 primal optimal vertex, and 4 dual optimal vertices. For the original $1e-6$ solutions, our crossover returns 76 primal-dual optimal vertices, 9 primal optimal vertices, 10 dual optimal vertices, and 5 near-optimal vertices. The inaccuracy of the original solution results in a loss of optimality in the vertex solution produced by the crossover procedure.

Significant enhancement of the solution sparsity is observed in both COPT and our crossover methods, with comparable levels of improvement. In several instances, our crossover performs differently from the classic simplex-based crossover in COPT, implying that our approach has the potential to serve as a viable alternative to the traditional method.

With the auxiliary LP heuristic for original $1e-8$ solutions, our crossover can successfully provide optimal or near-optimal vertex solutions in 93 instances. Compared to the results without the auxiliary LP heuristic, the number of OLS subproblems is significantly reduced, but extra LP solving time is introduced, resulting in a trade-off. Due to the instability of FOMs, the auxiliary LPs take a long time to solve for certain instances, sometimes even causing a timeout. However, despite failing to recover the optimal basis for these problems, the number of support of most primal solutions is still successfully reduced.

\section{Conclusions and Discussions} \label{sec:conclu}
In this paper, we formally derive the spiral behavior of PDHG for solving LP. The spiral behavior can be decomposed into a rotation and a forward movement that are orthogonal to each other. This explicit formula and geometric perspective provide insights into the recent findings of the algorithm, including the infeasibility detection, the two-stage behavior, the local linear rate, etc.

Inspired by the spiral of PDHG, we propose a new randomized crossover approach that is distinct from the traditional crossover approach based on the simplex algorithm. Our crossover moves along the PDHG spiral ray, which can be calculated via solving OLS subproblems. Our numerical experiments demonstrate the effectiveness of the proposed approach.

Looking ahead, several future directions are worth exploring to enhance this crossover scheme. Firstly, we can reuse the factorization in the OLS solving, potentially speeding up the algorithm significantly. One potential enhancement is implementing a warm start for PDHG, which could speed up auxiliary LPs, although a theoretical guarantee for this remains to be established. Another promising direction involves developing a GPU-friendly OLS solver, which would allow the crossover method to leverage GPU parallelization. Additionally, designing a strategic perturbation could direct the solution toward a desired vertex more efficiently. If this perturbation is well-calibrated to avoid making the LP unbounded or infeasible, the auxiliary LP heuristic could enable a purely PDHG-based, matrix-free crossover, making it particularly suitable for handling large-scale LPs.

\section*{Acknowledgments}
This work was started when the first author visited the second author at the University of Chicago, Booth School of Business. The authors also would like to thank Qi Huangfu for helpful discussions on the crossover algorithms in commercial LP solvers. The first author is partially supported by NSFC [Grant NSFC-72225009, 72394360, 72394365]. The second author is partially supported by AFOSR FA9550-24-1-0051.


\bibliographystyle{plainnat}
\bibliography{ref}

\newpage
\appendix

\section{Proof of Theorem \ref{thm:pdlp-spiral}} \label{apdx:prove-thm-pdlp-spiral}

\begin{proof}[Proof of Theorem \ref{thm:pdlp-spiral}.]
For each phase given PDHG the basis $(B, N)$, the PDHG operator can be rewritten as
\begin{equation*}
    \begin{bmatrix}
        \Ibf_{|B|}     &            &        \\
                       & \Ibf_{|N|} &        \\
        2\eta \Abf_B   &            & \Ibf_m
    \end{bmatrix}
    \begin{bmatrix}
        \xbf_B^{(k+1)} \\ \xbf_N^{(k+1)} \\ \ybf^{(k+1)}
    \end{bmatrix}
    =
    \begin{bmatrix}
        \Ibf_{|B|}    &            & \eta \Abf_B^\top \\
                      & \Ibf_{|N|} &                  \\
        \eta \Abf_B   &            & \Ibf_m
    \end{bmatrix}
    \begin{bmatrix}
        \xbf_B^{(k)} \\ \xbf_N^{(k)} \\ \ybf^{(k)}
    \end{bmatrix}
    +
    \begin{bmatrix}
        - \eta \cbf_B \\ \zerobf \\ \eta \bbf
    \end{bmatrix}.
\end{equation*}
Since the non-basic part $\xbf_N$ is always $\zerobf$, we focus on the basic part of PDHG iterates
\begin{equation*}
    \begin{bmatrix}
        \xbf_B^{(k+1)} \\ \ybf^{(k+1)}
    \end{bmatrix}
    =
    \begin{bmatrix}
        \Ibf_{|B|}    & \eta \Abf_B^\top                  \\
        - \eta \Abf_B & \Ibf_m - 2\eta^2 \Abf_B\Abf_B^\top
    \end{bmatrix}
    \begin{bmatrix}
        \xbf_B^{(k)} \\ \ybf^{(k)}
    \end{bmatrix}
    +
    \begin{bmatrix}
        - \eta \cbf_B \\ 2\eta^2\Abf_B\cbf_B + \eta \bbf
    \end{bmatrix}.
\end{equation*}
From \eqref{eq:sp-center} and \eqref{eq:sp-ray}, we have
\begin{equation*}
    \begin{bmatrix}
        \xbf_B^{(k+1)} - (\xbf_\vbf)_B \\ \ybf^{(k+1)} - \ybf_\vbf
    \end{bmatrix}
    =
    \Pbf_B
    \begin{bmatrix}
        \xbf_B^{(k)} - (\xbf_\vbf)_B \\ \ybf^{(k)} - \ybf_\vbf
    \end{bmatrix}
    +
    \begin{bmatrix}
        (\vbf_\xbf)_B \\ \vbf_\ybf
    \end{bmatrix}
\end{equation*}
where $\Pbf_B$ is defined in \eqref{eq:sp-operator}.

Note that
\begin{equation} \label{eq:Pv-v}
    \Pbf_B \vbf_B = \vbf_B, \ \vbf_B^\top \Pbf_B = \vbf_B^\top,
\end{equation}
we have \eqref{eq:iter} by induction.

Suppose we have the SVD decomposition of $\Abf_B$
\begin{equation} \label{eq:A-svd}
    \Abf_B = \Ubf\Sigmabf\Vbf^\top = 
    \begin{bmatrix}
        \Ubf_{r} & \Ubf_{m-r}
    \end{bmatrix}
    \begin{bmatrix}
        \sigma_1 \\
        & \ddots \\
        & & \sigma_r \\
        & & & & &
    \end{bmatrix}
    \begin{bmatrix}
        \Vbf_r & \Vbf_{|B| - r}
    \end{bmatrix}^\top,
\end{equation}
where $\Sigmabf$ has $r = \text{rank}(\Abf_B)$ nonzero diagonal elements, and $\Ubf, \Vbf$ are two unitary matrices. The first $r$ columns of $\Ubf$ and the remaining respectively form the basis of the image of $\Abf_B$ and the kernel of $\Abf_B^\top$. Analogously, the first $r$ columns of $\Vbf$ and the remaining respectively form the basis of the image of $\Abf_B^\top$ and the kernel of $\Abf_B$.
\begin{equation} \label{eq:kernel}
    \Abf_B \Vbf_{|B|-r} = \zerobf, \ \Abf_B^\top \Ubf_{m-r} = \zerobf.
\end{equation}

Replace $\Abf_B$ in \eqref{eq:sp-operator} by its SVD decomposition \eqref{eq:A-svd} and $\Pbf_B^k$ equivalently becomes
\begin{equation*}
    \Pbf_B^k = \begin{bmatrix}
        \Vbf \\
         & \Ubf
    \end{bmatrix}
    \begin{bmatrix}
        \Ibf_{|B|}     & \eta \Sigmabf^\top                         \\
        -\eta \Sigmabf & -2 \eta^2 \Sigmabf\Sigmabf^\top + \Ibf_{m}
    \end{bmatrix}^k
    \begin{bmatrix}
        \Vbf^\top   \\
         & \Ubf^\top
    \end{bmatrix}.
\end{equation*}

We can reorder the middle matrix as
\begin{equation*}
    \begin{bmatrix}
        [1]              &                  &   & & [\eta\sigma_1]                                        \\
                         & (1)              &   & &                         & (\eta\sigma_2)              \\
                         &                  & 1 & &                         &                         &   \\
        [- \eta\sigma_1] &                  &   & & [1 - 2\eta^2\sigma^2_1]                               \\
                         & (- \eta\sigma_2) &   & &                         & (1 - 2\eta^2\sigma^2_2)     \\
                         &                  &   & &                         &                         & 1 \\
    \end{bmatrix}
    \rightarrow
    \begin{bmatrix}
        \begin{bmatrix}
            1              & \eta\sigma_1         \\
            - \eta\sigma_1 & 1 - 2\eta^2\sigma^2_1
        \end{bmatrix}                                                                                 \\
                                                   & \begin{pmatrix}
                                                        1 & \eta\sigma_2                       \\
                                                        - \eta\sigma_2 & 1 - 2\eta^2\sigma^2_2
                                                     \end{pmatrix}                                    \\
                                                   &                                            & 1   \\
                                                   &                                            & & 1 \\
    \end{bmatrix}
\end{equation*}
to obtain several $2\times 2$ blocks and contingent ones on the diagonal. Moreover, since $\Abf_B$ is a submatrix of $\Abf$ and $\eta\|\Abf\|_2 < 1$, we still have
\begin{equation*}
    0 < \eta\sigma_i \leq \eta\|\Abf_B\|_2 \leq \eta\|\Abf\|_2 < 1, \ i = 1,\cdots,r.
\end{equation*}

For one $2\times2$ block $\Jbf = \begin{bmatrix}
        1            & \eta \sigma         \\
        -\eta \sigma & 1 - 2\eta^2\sigma^2
\end{bmatrix}$ with $\eta \sigma \in (0, 1)$, it is easy to prove that $\lim_{k\to+\infty}\Jbf^k = \zerobf$ by verifying that the modulus of its eigenvalues is strictly less than 1. Therefore, we have
\begin{equation} \label{eq:lim-Pk}
    \lim_{k\to\infty}\Pbf_B^k = 
    \begin{bmatrix}
        \Vbf_{|B|-r} \Vbf_{|B|-r}^\top \\
        & \Ubf_{m-r} \Ubf_{m-r}^\top
    \end{bmatrix}.
\end{equation}

For the vector $\zbf^{(0)} - \zbf_\vbf$, we have
\begin{equation*}
    \begin{aligned}
        \xbf_B^{(0)} &= \left[\Ibf - \Abf_B^\top(\Abf_B\Abf_B^\top)^\dagger\Abf_B\right]\xbf_B^{(0)} + \Abf_B^\top(\Abf_B\Abf_B^\top)^\dagger\Abf_B\xbf_B^{(0)} \\
        &= \proj_{\Abf_B\xbf = \zerobf}(\xbf_B^{(0)}) + \Abf_B^\top(\Abf_B\Abf_B^\top)^\dagger\Abf_B\xbf_B^{(0)}.
    \end{aligned}
\end{equation*}
and
\begin{equation*}
    \begin{aligned}
        (\xbf_\vbf)_B &= (\Abf_B^\top \Abf_B)^\dagger \Abf_B^\top \bbf + \proj_{\Abf_B\xbf = \zerobf}(\xbf_B^{(0)}) \\
        &= \Abf_B^\top (\Abf_B \Abf_B^\top)^\dagger \bbf + \proj_{\Abf_B\xbf = \zerobf}(\xbf_B^{(0)}).
    \end{aligned}
\end{equation*}
Here we utilize the property of Moore–Penrose inverse that $(\Abf_B^\top \Abf_B)^\dagger \Abf_B^\top = \Abf_B^\dagger = \Abf_B^\top (\Abf_B \Abf_B^\top)^\dagger$. Then we have
\begin{equation*}
    \begin{aligned}
        \xbf_B^{(0)} - (\xbf_\vbf)_B &= \Abf_B^\top (\Abf_B \Abf_B^\top)^\dagger (\Abf_B\xbf_B^{(0)} - \bbf) \\
        \ybf^{(0)} - \ybf_\vbf &= \Abf_B (\Abf_B^\top\Abf_B)^\dagger (\Abf_B^\top\ybf^{(0)} - \cbf_B),
    \end{aligned}
\end{equation*}
where the $\ybf^{(0)} - \ybf_\vbf$ part can be verified in a similar way.

To prove \eqref{eq:sp-cond-lim}, from \eqref{eq:lim-Pk} we have
\begin{equation} \label{eq:lim}
    \lim_{k\to \infty} \Pbf_B^k
    \begin{bmatrix}
        \xbf_B^{(0)} - (\xbf_\vbf)_B \\ \ybf^{(0)} - \ybf_\vbf
    \end{bmatrix} \\
    = 
    \begin{bmatrix}
        \Vbf_{|B|-r} \Vbf_{|B|-r}^\top \Abf_B^\top (\Abf_B \Abf_B^\top)^\dagger (\Abf_B\xbf_B^{(0)} - \bbf) \\
        \Ubf_{m-r} \Ubf_{m-r}^\top \Abf_B (\Abf_B^\top\Abf_B)^\dagger (\Abf_B^\top\ybf^{(0)} - \cbf_B)
    \end{bmatrix}
    = \zerobf,
\end{equation}
where the second equation comes from \eqref{eq:kernel}.

To prove \eqref{eq:sp-cond-orth}, for all $k$, we use \eqref{eq:Pv-v} to obtain
\begin{equation*}
    \begin{aligned}
          & \vbf_B^\top \Pbf_B^k
          \begin{bmatrix}
            \xbf_B^{(0)} - (\xbf_\vbf)_B \\ \ybf^{(0)} - \ybf_\vbf
          \end{bmatrix}
        = 
        \vbf_B^\top
        \begin{bmatrix}
            \xbf_B^{(0)} - (\xbf_\vbf)_B \\ \ybf^{(0)} - \ybf_\vbf
        \end{bmatrix} \\
        = &
        \begin{bmatrix}
            - \eta \left[\cbf_B - \Abf_B^\top(\Abf_B \Abf_B^\top)^\dagger \Abf_B \cbf_B\right] \\
            \eta \left[\bbf - \Abf_B(\Abf_B^\top \Abf_B)^\dagger \Abf_B^\top \bbf\right]
        \end{bmatrix}^\top
        \begin{bmatrix}
            \Abf_B^\top (\Abf_B \Abf_B^\top)^\dagger (\Abf_B\xbf_B^{(0)} - \bbf) \\
            \Abf_B (\Abf_B^\top\Abf_B)^\dagger (\Abf_B^\top\ybf^{(0)} - \cbf_B)
        \end{bmatrix} \\
        = &
        \begin{bmatrix}
            \eta \cbf_B \\
            - \eta \bbf
        \end{bmatrix}^\top
        \begin{bmatrix}
            \Ibf_{|B|} - \Abf_B^\top(\Abf_B \Abf_B^\top)^\dagger \Abf_B \\
                                                                    & \Ibf_{m} - \Abf_B(\Abf_B^\top \Abf_B)^\dagger \Abf_B^\top
        \end{bmatrix}
        \begin{bmatrix}
            \Abf_B^\top (\Abf_B \Abf_B^\top)^\dagger (\Abf_B\xbf_B^{(0)} - \bbf) \\
            \Abf_B (\Abf_B^\top\Abf_B)^\dagger (\Abf_B^\top\ybf^{(0)} - \cbf_B)
        \end{bmatrix} \\
        = &
        \begin{bmatrix}
            \eta \cbf_B \\
            - \eta \bbf
        \end{bmatrix}^\top
        \begin{bmatrix}
            \Abf_B^\top - \Abf_B^\top(\Abf_B^\top)^\dagger \Abf_B^\top \\
                                                                    & \Abf_B - \Abf_B\Abf_B^\dagger \Abf_B
        \end{bmatrix}
        \begin{bmatrix}
            (\Abf_B \Abf_B^\top)^\dagger (\Abf_B\xbf_B^{(0)} - \bbf) \\
            (\Abf_B^\top\Abf_B)^\dagger (\Abf_B^\top\ybf^{(0)} - \cbf_B)
        \end{bmatrix} \\
        = & \zerobf,
    \end{aligned}
\end{equation*}
where the last equation utilizes the property of Moore–Penrose inverse that $\Abf_B = \Abf_B\Abf_B^\dagger \Abf_B$.
\end{proof}

\section{Proof of Proposition \ref{prop:opt-improve}} \label{apdx:prove-prop-opt-improve}
\begin{proof}[Proof of Proposition \ref{prop:opt-improve}.]

From \eqref{eq:sp-center}, we have
\begin{equation} \label{eq:z_v}
    \begin{aligned}
        \Abf_B^\top \Abf_B (\xbf_\vbf)_B & = \Abf_B^\top \bbf \\
        \Abf_B \Abf_B^\top \ybf_\vbf     & = \Abf_B \cbf_B.
    \end{aligned}
\end{equation}
$(\xbf_\vbf)_B$ is the optimal solution to the OLS problem $\min \frac{1}{2}\|\Abf_B \xbf_B - \bbf\|_2^2$, while $\ybf_\vbf$ is the optimal solution to the OLS problem $\min \frac{1}{2}\|\Abf_B^\top \ybf - \cbf_B\|_2^2$.

From \eqref{eq:sp-ray}, we have
\begin{equation} \label{eq:idv}
    \begin{aligned}
        \Abf_B(\vbf_\xbf)_B  & = \zerobf \\
        \Abf_B^\top\vbf_\ybf & = \zerobf,
    \end{aligned}
\end{equation}
implying that the forward movement will not improve the primal or dual feasibility.

Suppose we have the SVD decomposition of $\Abf_B$ \eqref{eq:A-svd} and let $\Sigmabf_r$ denote the $r\times r$ nonzero-diagonal matrix in $\Sigmabf$, then
\begin{equation*}
    \begin{aligned}
        \cbf^\top \vbf_\xbf &= - \eta \cbf_B^\top\left[\cbf_B - \Abf_B^\top(\Abf_B \Abf_B^\top)^\dagger \Abf_B \cbf_B\right] \\
        &= - \eta \cbf_B^\top \left[ \Ibf_{|B|} - \Abf_B^\top(\Abf_B \Abf_B^\top)^\dagger \Abf_B \right] \cbf_B \\
        &= - \eta \cbf_B^\top \left[ \Ibf_{|B|} - \Vbf\Sigmabf^\top\Ubf^\top(\Ubf\Sigmabf\Sigmabf^\top \Ubf^\top)^\dagger \Ubf\Sigmabf\Vbf^\top \right] \cbf_B \\
        &= - \eta \cbf_B^\top \left[ \Ibf_{|B|} - \Vbf\Sigmabf^\top (\Sigmabf\Sigmabf^\top)^\dagger \Sigmabf\Vbf^\top \right] \cbf_B \\
        &= - \eta \cbf_B^\top \left[ \Ibf_{|B|} - \Vbf_r\Sigmabf_r \Sigmabf_r^{-2} \Sigmabf_r\Vbf_r^\top \right] \cbf_B \\
        &= - \eta \cbf_B^\top \left[ \Ibf_{|B|} - \Vbf_r\Vbf_r^\top \right] \cbf_B \\
        &= - \eta \cbf_B^\top \left[ \Vbf_{|B|-r}\Vbf_{|B|-r}^\top \right] \cbf_B \\
        &\leq 0,
    \end{aligned}
\end{equation*}
and similarly
\begin{equation*}
    \bbf^\top \vbf_\ybf = \eta \bbf^\top \left[\Ubf_{m-r}\Ubf_{m-r}^\top \right] \bbf \geq 0.
\end{equation*}
The last inequality of $\cbf^\top\vbf_\xbf \leq 0$ is activated if and only if $\Vbf_{|B|-r}^\top\cbf_B=\zerobf$, i.e., $\cbf_B$ belongs to the image of $\Abf_B^\top$. From \eqref{eq:sp-ray}, $\cbf_B$ belongs to the image of $\Abf_B^\top$ if and only if $(\vbf_\xbf)_B$ belongs to the image of $\Abf_B^\top$. Moreover, \eqref{eq:idv} tells us $(\vbf_\xbf)_B$ belongs to the kernel of $\Abf_B$, so $(\vbf_\xbf)_B$ belongs to the image of $\Abf_B^\top$ if and only if $(\vbf_\xbf)_B = \zerobf$.
Similarly, $\bbf^\top \vbf_\ybf \geq 0$ is activated if and only if $\Ubf_{|B|-r}^\top\bbf=\zerobf$, i.e., $\bbf$ belongs to the image of $\Abf_B$, or if and only if $\vbf_\ybf = \zerobf$. Combining the primal and dual parts completes the proof.
\end{proof}

\section{Proof of Proposition \ref{prop:fix-lp}} \phantomsection\label{apdx:prove-prop-fix-lp}
\begin{proof}[Proof of Proposition \ref{prop:fix-lp}.]
    The KKT condition of \eqref{prob:lp} is
    \begin{equation} \label{eq:lp-optimal}
        \begin{aligned}
            & \Abf \xbf = \bbf, \ \xbf \geq \zerobf \\
            & \cbf - \Abf^\top \ybf \geq \zerobf \\
            & \cbf^\top \xbf - \bbf^\top \ybf = 0.
        \end{aligned}
    \end{equation}
    Since $(\xbf,\ybf)$ is optimal, we have $B \subseteq D$ and $N \subseteq E$. The duality gap of $(\tilde{\xbf}, \tilde{\ybf})$ is
    \begin{equation*}
        \cbf^\top \tilde{\xbf} - \bbf^\top \tilde{\ybf} = \cbf_B^\top \tilde{\xbf}_B - \tilde{\xbf}_B^\top \Abf_B^\top \tilde{\ybf} = \cbf_D^\top \tilde{\xbf}_D - \tilde{\xbf}_D ^\top \Abf_D^\top \tilde{\ybf} = 0.
    \end{equation*}
    The last equality is from $\cbf_D - \Abf_D^\top \tilde{\ybf} = \zerobf$.
    Together with \eqref{eq:fix-lp}, $(\tilde{\xbf}, \tilde{\ybf})$ satisfies \eqref{eq:lp-optimal}.
\end{proof}

\section{Proof of Proposition \ref{prop:feas-point}} \phantomsection\label{apdx:prove-prop-feas-point}
\begin{proof}[Proof of Proposition \ref{prop:feas-point}.]
    To prove $\xbf_\vbf = \xbf^{(0)}$ if $\Abf_B\xbf^{(0)}_B = \bbf$, we have $(\xbf_\vbf)_N = \xbf^{(0)}_N = \zerobf$ and
    \begin{equation*}
        \begin{aligned}
            (\xbf_\vbf)_B &= (\Abf_B^\top \Abf_B)^\dagger \Abf_B^\top \bbf + \proj_{\Abf_B\xbf = \zerobf}(\xbf_B^{(0)}) \\
            &= \Abf_B^\top (\Abf_B \Abf_B^\top)^\dagger \bbf + \proj_{\Abf_B\xbf = \zerobf}(\xbf_B^{(0)}) \\
            &= \Abf_B^\top (\Abf_B \Abf_B^\top)^\dagger \Abf_B \xbf^{(0)}_B + \left[ \Ibf - \Abf_B^\top (\Abf_B \Abf_B^\top)^\dagger \Abf_B \right] \xbf^{(0)}_B \\
            &= \xbf^{(0)}_B.
        \end{aligned}
    \end{equation*}
    The second equation utilizes the property of Moore–Penrose inverse that $(\Abf_B^\top \Abf_B)^\dagger \Abf_B^\top = \Abf_B^\dagger = \Abf_B^\top (\Abf_B \Abf_B^\top)^\dagger$. Analogously, if $\Abf_B^\top \ybf^{(0)} = \cbf_B$, we also have $\ybf_\vbf = \ybf^{(0)}$.
\end{proof}

\section{Proof of Theorem \ref{thm:ols-number}} \phantomsection\label{apdx:prove-thm-ols-number}

\begin{proof}[Proof of Theorem \ref{thm:ols-number}.]
For the primal push phase, at the beginning of each primal push, we have $B = \{i:\xbf_i > 0\}$ and then
\[
    |B| \geq \text{rank}(\Abf_B)
\]
because the rank of a matrix will not exceed the number of its columns. Moreover, since $\Abf$ is full row rank, there exist $m - \text{rank}(\Abf_B)$ columns in $\Abf$ that are linearly independent of columns in $\Abf_B$, and thus
\[
    |B| \leq n - [m - \text{rank}(\Abf_B)].
\]
Together we have
\begin{equation} \label{eq:B-rk}
     0 \leq |B| - \text{rank}(\Abf_B) \leq n - m.
\end{equation}
Similar to Appendix \ref{apdx:prove-prop-opt-improve}, we have SVD decomposition of $\Abf_B$ \eqref{eq:A-svd} and
\[
\begin{aligned}
    (\deltabf_\xbf)_B & = - \left[ \tilde{\cbf}_B - \Abf_B^\top (\Abf_B\Abf_B^\top)^\dagger\Abf_B\tilde{\cbf}_B \right] \\
    & = - \left[ \Ibf_{|B|} - \Abf_B^\top (\Abf_B\Abf_B^\top)^\dagger\Abf_B \right] \tilde{\cbf}_B \\
    & = - \left[ \Vbf_{|B|-r}\Vbf_{|B|-r}^\top \right] \tilde{\cbf}_B,
\end{aligned}
\]
where $r = \text{rank}(\Abf_B)$. Since we randomly perturb the cost vector $\tilde{\cbf}_B$, we have $(\deltabf_\xbf)_B \neq \zerobf$ w.p. 1 if and only if $|B| > \text{rank}(\Abf_B)$. If $(\deltabf_\xbf)_B = \zerobf$, the primal push phase is over. Otherwise, if $(\deltabf_\xbf)_B \geq \zerobf$, we can replace $(\deltabf_\xbf)_B$ by $-(\deltabf_\xbf)_B$ to ensure that the moving direction has a strictly negative component.

Now consider the case when $(\deltabf_\xbf)_B \neq \zerobf$ has a strictly negative component. For $\theta = \min_{\{i\in B: (\deltabf_\xbf)_i < 0\}}\{-\frac{\xbf_i}{(\deltabf_\xbf)_i}\}$ and all $j\in J = \argmin_{\{i\in B: (\deltabf_\xbf)_i < 0\}}\{-\frac{\xbf_i}{(\deltabf_\xbf)_i}\}$, we have $[\xbf_B + \theta (\deltabf_\xbf)_B]_j=0$, which indicates that these columns will be removed from $\Abf_B$. Since $\Abf_B (\deltabf_\xbf)_B = \zerobf$ and $(\deltabf_\xbf)_j < 0$ for $j\in J$, at least one column to be removed is linearly dependent on the other columns, and thus each primal push will decrease $|B| - \text{rank}(\Abf_B)$ by at least 1. From \eqref{eq:B-rk}, primal push will terminate within $n-m$ steps w.p. 1, which also means at most $n-m$ OLS subproblems are required in the primal push phase.

For the dual push phase, we set $D = \{i: \cbf_i - \Abf_i^\top \ybf = 0\}$. Similarly, we have $\text{rank}(\Abf_D)\leq m$ and $\text{rank}(\Abf_D)$ strictly increases by at least 1 in each dual push. Thus, dual push will terminate within $m$ steps w.p. 1, which also means at most $m$ OLS subproblems are required in the dual push phase.

Therefore, our crossover requires a total of $(n-m)+m = n$ OLS subproblems.
\end{proof}

\section{Numerical Results} \phantomsection\label{apdx:netlib-table}

The numerical results for PDLP solutions with accuracy of $1e-8$ are shown in Table \ref{tbl:netlib-1e-8}. ``nRows'' and ``nCols'' are the number of rows and columns of $\Abf$ in \eqref{prob:general-lp}. ``\# supp. PDLP'' is the number of support (i.e., components not on their bounds) of the initial solution, ``\# supp. COPT'' and ``\# supp. cross'' are the numbers of support of vertices from the crossover algorithm in COPT commercial solver and our crossover approach. We also report the total crossover time, along with the runtimes of different components, in the ``time (sec)'' columns. ``nOLS'' columns are the number of OLS subproblems in the primal push and dual push phase of the crossover.

For robustness comparison, we also provide the crossover results for PDLP solutions with moderate accuracy of $1e-6$ in Table \ref{tbl:netlib-1e-6}. The results with the auxiliary LP heuristic for original $1e-8$ solutions are presented in Table \ref{tbl:netlib-lp-1e-8}.

{\small
\begin{longtable}{ |c|c|c|c|c|c|c|c|c|c|c|c| }

\toprule
 \multirow{2}{*}{prob} & \multirow{2}{*}{nRows} & \multirow{2}{*}{nCols} & \multicolumn{3}{c|}{\# supp.} & \multicolumn{4}{c|}{time (sec)} & \multicolumn{2}{c|}{nOLS} \\
\cline{4-12}
 & & & PDLP & COPT & cross & cross & primal & dual & lin. ind. & primal & dual \\
\midrule

\endfirsthead

\toprule
 \multirow{2}{*}{prob} & \multirow{2}{*}{nRows} & \multirow{2}{*}{nCols} & \multicolumn{3}{c|}{\# supp.} & \multicolumn{4}{c|}{time (sec)} & \multicolumn{2}{c|}{nOLS} \\
\cline{4-12}
 & & & PDLP & COPT & cross & cross & primal & dual & lin. ind. & primal & dual \\
\midrule

\endhead
\endfoot
25fv47 & 821 & 1571 & 600 & 583 & 584 & 0.20 & 0.03 & 0.07 & 0.03 & 17 & 50 \\
80bau3b & 2262 & 9799 & 1851 & 1758 & 1752 & \textcolor{blue}{1.22d} & 0.15 & 0.34 & 0.73 & 86 & 205 \\
adlittle & 56 & 97 & 61 & 45 & 45 & 0.01 & 0.01 & 0.00 & 0.00 & 17 & 2 \\
afiro & 27 & 32 & 14 & 13 & 13 & 0.00 & 0.00 & 0.00 & 0.00 & 2 & 2 \\
agg & 488 & 163 & 62 & 57 & 57 & 0.01 & 0.00 & 0.00 & 0.00 & 6 & 2 \\
agg2 & 516 & 302 & 141 & 120 & 121 & 0.02 & 0.01 & 0.00 & 0.00 & 26 & 4 \\
agg3 & 516 & 302 & 144 & 124 & 124 & 0.01 & 0.01 & 0.00 & 0.00 & 26 & 2 \\
bandm & 305 & 472 & 306 & 294 & 294 & 0.01 & 0.00 & 0.00 & 0.00 & 2 & 2 \\
beaconfd & 173 & 262 & 89 & 89 & 89 & 0.00 & 0.00 & 0.00 & 0.00 & 2 & 2 \\
blend & 74 & 83 & 56 & 54 & 54 & 0.00 & 0.00 & 0.00 & 0.00 & 3 & 3 \\
bnl1 & 643 & 1175 & 689 & 451 & 449 & 0.34 & 0.31 & 0.01 & 0.02 & 230 & 5 \\
bnl2 & 2324 & 3489 & 1519 & 1171 & 1164 & 1.92 & 0.74 & 0.17 & 1.01 & 358 & 78 \\
boeing1 & 351 & 384 & 207 & 196 & 195 & 0.02 & 0.01 & 0.01 & 0.00 & 15 & 26 \\
boeing2 & 166 & 143 & 69 & 55 & 57 & 0.01 & 0.00 & 0.01 & 0.00 & 14 & 30 \\
bore3d & 233 & 315 & 126 & 126 & 126 & 0.00 & 0.00 & 0.00 & 0.00 & 2 & 2 \\
brandy & 220 & 249 & 135 & 134 & 134 & 0.01 & 0.00 & 0.00 & 0.00 & 2 & 2 \\
capri & 271 & 353 & 246 & 220 & 220 & 0.02 & 0.02 & 0.00 & 0.00 & 35 & 8 \\
cre-a & 3516 & 4067 & 579 & 492 & 491 & 11.30 & 0.16 & 10.27 & 0.86 & 112 & 1026 \\
cre-c & 3068 & 3678 & 590 & 508 & 504 & 9.00 & 0.15 & 8.08 & 0.77 & 99 & 766 \\
cycle & 1903 & 2857 & 994 & 224 & 445 & 1.04 & 0.34 & 0.45 & 0.25 & 401 & 384 \\
czprob & 929 & 3523 & 924 & 866 & 866 & 0.12 & 0.09 & 0.00 & 0.03 & 59 & 2 \\
d2q06c & 2171 & 5167 & 1612 & 1543 & 1543 & 1.54 & 0.32 & 0.56 & 0.66 & 31 & 121 \\
d6cube & 415 & 6184 & 136 & 115 & 110 & 0.58 & 0.02 & 0.54 & 0.01 & 21 & 284 \\
degen2 & 444 & 534 & 207 & 207 & 207 & 0.28 & 0.00 & 0.25 & 0.03 & 2 & 206 \\
degen3 & 1503 & 1818 & 640 & 637 & 638 & 8.14 & 0.01 & 7.33 & 0.79 & 3 & 776 \\
e226 & 223 & 282 & 127 & 127 & 127 & 0.01 & 0.00 & 0.00 & 0.00 & 2 & 5 \\
etamacro & 400 & 688 & 292 & 271 & 272 & 0.03 & 0.01 & 0.02 & 0.01 & 22 & 50 \\
fffff800 & 524 & 854 & 357 & 312 & 307 & 0.04 & 0.02 & 0.01 & 0.01 & 52 & 36 \\
finnis & 497 & 614 & 259 & 227 & 230 & 0.05 & 0.02 & 0.02 & 0.01 & 48 & 54 \\
fit1d & 24 & 1026 & 12 & 12 & 12 & 0.00 & 0.00 & 0.00 & 0.00 & 2 & 0 \\
fit1p & 627 & 1677 & 634 & 627 & 627 & 0.00 & 0.00 & 0.00 & 0.00 & 2 & 0 \\
fit2d & 25 & 10500 & 22 & 20 & 20 & 0.01 & 0.00 & 0.00 & 0.00 & 3 & 0 \\
fit2p & 3000 & 13525 & 3004 & 2997 & 2997 & 0.04 & 0.00 & 0.01 & 0.03 & 2 & 4 \\
forplan & 161 & 421 & 83 & 83 & 83 & 0.01 & 0.00 & 0.01 & 0.00 & 2 & 25 \\
ganges & 1309 & 1681 & 1278 & 1175 & 1174 & 0.18 & 0.08 & 0.00 & 0.09 & 103 & 6 \\
gfrd-pnc & 616 & 1092 & 349 & 336 & 336 & 0.09 & 0.00 & 0.06 & 0.02 & 2 & 125 \\
greenbeb & 2392 & 5405 & 1048 & 937 & 936 & 2.77 & 0.25 & 1.47 & 1.04 & 104 & 404 \\
grow15 & 300 & 645 & 533 & 299 & 300 & 0.24 & 0.24 & 0.00 & 0.00 & 234 & 0 \\
grow22 & 440 & 946 & 849 & 440 & 440 & 0.62 & 0.62 & 0.00 & 0.00 & 410 & 0 \\
grow7 & 140 & 301 & 237 & 140 & 140 & 0.05 & 0.05 & 0.00 & 0.00 & 98 & 0 \\
israel & 174 & 142 & 80 & 70 & 69 & 0.01 & 0.01 & 0.00 & 0.00 & 20 & 0 \\
kb2 & 43 & 41 & 27 & 27 & 27 & 0.00 & 0.00 & 0.00 & 0.00 & 2 & 0 \\
ken-07 & 2426 & 3602 & 2236 & 2234 & 2234 & 1.27 & 0.00 & 0.05 & 1.22 & 2 & 65 \\
lotfi & 153 & 308 & 126 & 99 & 99 & 0.01 & 0.01 & 0.00 & 0.00 & 29 & 2 \\
maros-r7 & 3136 & 9408 & 3136 & 3136 & 3136 & 0.09 & 0.06 & 0.00 & 0.01 & 2 & 0 \\
maros & 846 & 1443 & 345 & 340 & 339 & 0.32 & 0.01 & 0.27 & 0.05 & 7 & 162 \\
modszk1 & 687 & 1620 & 666 & 666 & 666 & 0.09 & 0.02 & 0.03 & 0.03 & 2 & 20 \\
nesm & 662 & 2923 & 726 & 550 & 543 & 0.26 & 0.25 & 0.00 & 0.00 & 188 & 2 \\
osa-07 & 1118 & 23949 & 357 & 355 & 355 & 0.05 & 0.01 & 0.01 & 0.02 & 8 & 6 \\
osa-14 & 2337 & 52460 & 873 & 781 & 781 & 0.52 & 0.35 & 0.03 & 0.12 & 93 & 6 \\
osa-30 & 4350 & 100024 & 1733 & 1536 & 1536 & 2.00 & 1.47 & 0.06 & 0.46 & 207 & 6 \\
pds-02 & 2953 & 7535 & 1379 & 332 & 331 & 1.68 & 0.79 & 0.38 & 0.51 & 432 & 396 \\
perold & 625 & 1376 & 580 & 546 & 549 & \textcolor{blue}{0.16d} & 0.08 & 0.05 & 0.02 & 38 & 27 \\
pilot.ja & 940 & 1988 & 789 & 681 & 678 & \textcolor{blue}{0.58d} & 0.40 & 0.09 & 0.09 & 92 & 30 \\
pilot & 1441 & 3652 & 1299 & 1287 & 1282 & \textcolor{blue}{0.58p} & 0.28 & 0.09 & 0.20 & 22 & 8 \\
pilot4 & 410 & 1000 & 374 & 364 & 364 & 0.04 & 0.02 & 0.01 & 0.01 & 14 & 7 \\
pilot87 & 2030 & 4883 & 1856 & 1840 & 1840 & 1.10 & 0.45 & 0.31 & 0.31 & 17 & 15 \\
pilotnov & 975 & 2172 & 1809 & 708 & 650 & \textcolor{blue}{1.83d} & 1.71 & 0.00 & 0.12 & 1163 & 2 \\
qap12 & 3192 & 8856 & 2940 & 2462 & 2466 & 66.60 & 32.42 & 8.85 & 25.33 & 435 & 216 \\
qap8 & 912 & 1632 & 656 & 466 & 442 & 0.82 & 0.26 & 0.20 & 0.36 & 56 & 112 \\
recipe & 91 & 180 & 24 & 24 & 24 & 0.00 & 0.00 & 0.00 & 0.00 & 2 & 5 \\
sc105 & 105 & 103 & 85 & 85 & 85 & 0.00 & 0.00 & 0.00 & 0.00 & 2 & 6 \\
sc205 & 205 & 203 & 184 & 184 & 184 & 0.01 & 0.00 & 0.00 & 0.00 & 2 & 6 \\
sc50a & 50 & 48 & 42 & 42 & 42 & 0.00 & 0.00 & 0.00 & 0.00 & 2 & 2 \\
sc50b & 50 & 48 & 48 & 48 & 48 & 0.01 & 0.00 & 0.00 & 0.01 & 2 & 2 \\
scagr25 & 471 & 500 & 317 & 307 & 307 & 0.01 & 0.00 & 0.00 & 0.00 & 2 & 2 \\
scagr7 & 129 & 140 & 98 & 97 & 97 & 0.00 & 0.00 & 0.00 & 0.00 & 2 & 0 \\
scfxm1 & 330 & 457 & 248 & 231 & 231 & 0.02 & 0.01 & 0.01 & 0.00 & 13 & 21 \\
scfxm2 & 660 & 914 & 501 & 473 & 471 & 0.05 & 0.02 & 0.02 & 0.01 & 24 & 27 \\
scfxm3 & 990 & 1371 & 754 & 711 & 711 & 0.11 & 0.04 & 0.03 & 0.04 & 35 & 33 \\
scorpion & 388 & 358 & 245 & 245 & 245 & 0.01 & 0.00 & 0.00 & 0.01 & 2 & 2 \\
scrs8 & 490 & 1169 & 276 & 276 & 276 & 0.02 & 0.00 & 0.00 & 0.01 & 2 & 10 \\
scsd1 & 77 & 760 & 31 & 7 & 12 & 0.02 & 0.00 & 0.01 & 0.00 & 13 & 60 \\
scsd6 & 147 & 1350 & 182 & 63 & 66 & 0.06 & 0.04 & 0.01 & 0.00 & 85 & 52 \\
scsd8 & 397 & 2750 & 551 & 147 & 354 & 0.14 & 0.12 & 0.01 & 0.00 & 175 & 23 \\
sctap1 & 300 & 480 & 263 & 164 & 164 & 0.07 & 0.05 & 0.01 & 0.00 & 102 & 46 \\
sctap2 & 1090 & 1880 & 811 & 562 & 561 & 0.26 & 0.11 & 0.06 & 0.09 & 231 & 143 \\
sctap3 & 1480 & 2480 & 1038 & 731 & 730 & 0.51 & 0.16 & 0.13 & 0.21 & 279 & 270 \\
seba & 515 & 1028 & 438 & 438 & 438 & 0.00 & 0.00 & 0.00 & 0.00 & 2 & 2 \\
share1b & 117 & 225 & 95 & 94 & 94 & 0.00 & 0.00 & 0.00 & 0.00 & 2 & 0 \\
share2b & 96 & 79 & 52 & 48 & 48 & 0.01 & 0.00 & 0.00 & 0.00 & 7 & 11 \\
shell & 536 & 1775 & 391 & 383 & 383 & 0.03 & 0.00 & 0.01 & 0.02 & 2 & 48 \\
ship04l & 402 & 2118 & 261 & 260 & 260 & 0.01 & 0.00 & 0.01 & 0.01 & 2 & 24 \\
ship04s & 402 & 1458 & 281 & 280 & 280 & 0.02 & 0.00 & 0.01 & 0.01 & 2 & 16 \\
ship08l & 778 & 4283 & 422 & 422 & 422 & 0.08 & 0.00 & 0.03 & 0.05 & 2 & 34 \\
ship08s & 778 & 2387 & 447 & 447 & 447 & 0.06 & 0.00 & 0.01 & 0.05 & 2 & 22 \\
ship12l & 1151 & 5427 & 707 & 706 & 706 & 0.22 & 0.01 & 0.04 & 0.16 & 2 & 57 \\
ship12s & 1151 & 2763 & 728 & 728 & 728 & 0.19 & 0.00 & 0.03 & 0.16 & 2 & 47 \\
sierra & 1227 & 2036 & 373 & 361 & 361 & 0.18 & 0.00 & 0.04 & 0.14 & 5 & 92 \\
stair & 356 & 467 & 350 & 349 & 349 & 0.01 & 0.00 & 0.00 & 0.00 & 2 & 2 \\
standata & 359 & 1075 & 71 & 50 & 50 & 0.05 & 0.00 & 0.04 & 0.00 & 10 & 146 \\
standgub & 361 & 1184 & 71 & 50 & 50 & 0.04 & 0.00 & 0.03 & 0.00 & 10 & 146 \\
standmps & 467 & 1075 & 190 & 174 & 174 & 0.05 & 0.01 & 0.05 & 0.00 & 15 & 144 \\
stocfor1 & 117 & 111 & 69 & 69 & 69 & 0.00 & 0.00 & 0.00 & 0.00 & 2 & 3 \\
stocfor2 & 2157 & 2031 & 1267 & 1267 & 1267 & 0.85 & 0.00 & 0.13 & 0.72 & 2 & 123 \\
truss & 1000 & 8806 & 802 & 691 & 698 & 0.87 & 0.18 & 0.56 & 0.13 & 102 & 300 \\
tuff & 333 & 587 & 165 & 122 & 122 & 0.04 & 0.02 & 0.02 & 0.00 & 43 & 60 \\
vtp.base & 198 & 203 & 55 & 55 & 55 & 0.00 & 0.00 & 0.00 & 0.00 & 2 & 2 \\
wood1p & 244 & 2594 & 39 & 39 & 39 & 0.14 & 0.00 & 0.14 & 0.00 & 2 & 134 \\
woodw & 1098 & 8405 & 696 & 550 & 550 & 0.60 & 0.27 & 0.19 & 0.13 & 146 & 81 \\

\bottomrule
\caption{Crossover results ($1e-8$ solutions) for the 100 NETLIB instances. ``p'' means the basis is only primal optimal (dual infeasible). ``d'' means the basis is only dual optimal (primal feasible). ``a'' means the basis is neither primal nor dual infeasible. All OLS subproblems are solved by direct QR factorization.} \label{tbl:netlib-1e-8}
\end{longtable}
}

{\small
\begin{longtable}{ |c|c|c|c|c|c|c|c|c|c|c|c| }

\toprule
 \multirow{2}{*}{prob} & \multirow{2}{*}{nRows} & \multirow{2}{*}{nCols} & \multicolumn{3}{c|}{\# supp.} & \multicolumn{4}{c|}{time (sec)} & \multicolumn{2}{c|}{nOLS} \\
\cline{4-12}
 & & & PDLP & COPT & cross & cross & primal & dual & lin. ind. & primal & dual \\
\midrule

\endfirsthead

\toprule
 \multirow{2}{*}{prob} & \multirow{2}{*}{nRows} & \multirow{2}{*}{nCols} & \multicolumn{3}{c|}{\# supp.} & \multicolumn{4}{c|}{time (sec)} & \multicolumn{2}{c|}{nOLS} \\
\cline{4-12}
 & & & PDLP & COPT & cross & cross & primal & dual & lin. ind. & primal & dual \\
\midrule

\endhead
\endfoot
25fv47 & 821 & 1571 & 601 & 583 & 584 & 0.13 & 0.02 & 0.08 & 0.03 & 17 & 50 \\
80bau3b & 2262 & 9799 & 1852 & 1759 & 1742 & \textcolor{blue}{1.28a} & 0.16 & 0.37 & 0.74 & 69 & 225 \\
adlittle & 56 & 97 & 64 & 45 & 45 & \textcolor{blue}{0.01p} & 0.01 & 0.00 & 0.00 & 17 & 2 \\
afiro & 27 & 32 & 14 & 13 & 13 & 0.00 & 0.00 & 0.00 & 0.00 & 2 & 2 \\
agg & 488 & 163 & 62 & 57 & 57 & 0.01 & 0.00 & 0.00 & 0.00 & 6 & 2 \\
agg2 & 516 & 302 & 141 & 120 & 121 & 0.01 & 0.01 & 0.00 & 0.00 & 26 & 4 \\
agg3 & 516 & 302 & 144 & 124 & 124 & 0.02 & 0.02 & 0.00 & 0.00 & 26 & 2 \\
bandm & 305 & 472 & 326 & 294 & 294 & 0.01 & 0.00 & 0.00 & 0.00 & 2 & 2 \\
beaconfd & 173 & 262 & 89 & 89 & 89 & 0.00 & 0.00 & 0.00 & 0.00 & 2 & 2 \\
blend & 74 & 83 & 56 & 54 & 54 & 0.00 & 0.00 & 0.00 & 0.00 & 3 & 3 \\
bnl1 & 643 & 1175 & 689 & 451 & 449 & 0.35 & 0.32 & 0.01 & 0.02 & 230 & 6 \\
bnl2 & 2324 & 3489 & 1523 & 1170 & 1165 & \textcolor{blue}{1.90p} & 0.68 & 0.22 & 0.99 & 359 & 79 \\
boeing1 & 351 & 384 & 210 & 195 & 196 & \textcolor{blue}{0.03a} & 0.01 & 0.02 & 0.00 & 18 & 26 \\
boeing2 & 166 & 143 & 69 & 55 & 57 & 0.01 & 0.00 & 0.00 & 0.00 & 14 & 30 \\
bore3d & 233 & 315 & 127 & 126 & 126 & 0.00 & 0.00 & 0.00 & 0.00 & 2 & 2 \\
brandy & 220 & 249 & 136 & 134 & 134 & 0.01 & 0.00 & 0.00 & 0.00 & 2 & 2 \\
capri & 271 & 353 & 246 & 220 & 220 & 0.02 & 0.02 & 0.00 & 0.00 & 35 & 8 \\
cre-a & 3516 & 4067 & 578 & 493 & 490 & 10.96 & 0.16 & 9.91 & 0.89 & 110 & 1026 \\
cre-c & 3068 & 3678 & 591 & 508 & 503 & 8.69 & 0.14 & 7.76 & 0.80 & 99 & 766 \\
cycle & 1903 & 2857 & 1032 & 358 & 229 & 1.05 & 0.33 & 0.37 & 0.34 & 403 & 390 \\
czprob & 929 & 3523 & 925 & 866 & 866 & \textcolor{blue}{0.21p} & 0.17 & 0.00 & 0.03 & 60 & 2 \\
d2q06c & 2171 & 5167 & 1645 & 1542 & 1542 & \textcolor{blue}{1.40d} & 0.13 & 0.57 & 0.69 & 31 & 125 \\
d6cube & 415 & 6184 & 138 & 115 & 110 & 0.60 & 0.02 & 0.56 & 0.01 & 21 & 284 \\
degen2 & 444 & 534 & 207 & 207 & 207 & 0.28 & 0.00 & 0.25 & 0.03 & 2 & 206 \\
degen3 & 1503 & 1818 & 645 & 637 & 637 & 7.90 & 0.01 & 7.09 & 0.79 & 3 & 776 \\
e226 & 223 & 282 & 129 & 127 & 127 & 0.01 & 0.00 & 0.00 & 0.00 & 2 & 5 \\
etamacro & 400 & 688 & 299 & 269 & 269 & \textcolor{blue}{0.04d} & 0.02 & 0.01 & 0.01 & 25 & 48 \\
fffff800 & 524 & 854 & 367 & 313 & 311 & \textcolor{blue}{0.05p} & 0.02 & 0.01 & 0.01 & 53 & 36 \\
finnis & 497 & 614 & 246 & 228 & 230 & \textcolor{blue}{0.05d} & 0.02 & 0.02 & 0.01 & 47 & 64 \\
fit1d & 24 & 1026 & 12 & 12 & 12 & 0.00 & 0.00 & 0.00 & 0.00 & 2 & 0 \\
fit1p & 627 & 1677 & 634 & 627 & 627 & 0.00 & 0.00 & 0.00 & 0.00 & 2 & 0 \\
fit2d & 25 & 10500 & 22 & 20 & 20 & 0.01 & 0.01 & 0.00 & 0.00 & 3 & 0 \\
fit2p & 3000 & 13525 & 3007 & 2997 & 2996 & \textcolor{blue}{0.07d} & 0.00 & 0.01 & 0.05 & 2 & 5 \\
forplan & 161 & 421 & 85 & 83 & 83 & \textcolor{blue}{0.01p} & 0.00 & 0.01 & 0.00 & 3 & 25 \\
ganges & 1309 & 1681 & 1278 & 1175 & 1173 & \textcolor{blue}{0.18d} & 0.08 & 0.00 & 0.10 & 103 & 7 \\
gfrd-pnc & 616 & 1092 & 349 & 336 & 336 & 0.17 & 0.00 & 0.15 & 0.02 & 2 & 125 \\
greenbeb & 2392 & 5405 & 1074 & 939 & 935 & \textcolor{blue}{3.48p} & 0.25 & 2.21 & 1.02 & 122 & 598 \\
grow15 & 300 & 645 & 533 & 299 & 300 & 0.35 & 0.35 & 0.00 & 0.00 & 234 & 0 \\
grow22 & 440 & 946 & 849 & 440 & 440 & 0.61 & 0.61 & 0.00 & 0.00 & 410 & 0 \\
grow7 & 140 & 301 & 237 & 140 & 140 & 0.05 & 0.05 & 0.00 & 0.00 & 98 & 0 \\
israel & 174 & 142 & 79 & 69 & 68 & \textcolor{blue}{0.04p} & 0.01 & 0.00 & 0.01 & 23 & 0 \\
kb2 & 43 & 41 & 30 & 27 & 27 & 0.00 & 0.00 & 0.00 & 0.00 & 2 & 0 \\
ken-07 & 2426 & 3602 & 2235 & 2234 & 2234 & 1.23 & 0.00 & 0.04 & 1.18 & 2 & 65 \\
lotfi & 153 & 308 & 129 & 99 & 99 & 0.01 & 0.01 & 0.00 & 0.00 & 29 & 2 \\
maros-r7 & 3136 & 9408 & 3135 & 3136 & 3136 & 0.24 & 0.03 & 0.11 & 0.08 & 2 & 2 \\
maros & 846 & 1443 & 351 & 340 & 339 & \textcolor{blue}{0.23p} & 0.01 & 0.17 & 0.05 & 8 & 161 \\
modszk1 & 687 & 1620 & 666 & 666 & 666 & 0.05 & 0.00 & 0.02 & 0.03 & 2 & 20 \\
nesm & 662 & 2923 & 758 & 553 & 538 & \textcolor{blue}{0.17a} & 0.17 & 0.00 & 0.00 & 218 & 2 \\
osa-07 & 1118 & 23949 & 358 & 355 & 355 & 0.06 & 0.01 & 0.02 & 0.02 & 8 & 6 \\
osa-14 & 2337 & 52460 & 882 & 781 & 781 & 0.50 & 0.34 & 0.03 & 0.13 & 92 & 6 \\
osa-30 & 4350 & 100024 & 1741 & 1536 & 1536 & 2.01 & 1.48 & 0.06 & 0.46 & 207 & 6 \\
pds-02 & 2953 & 7535 & 1389 & 332 & 319 & 1.77 & 0.78 & 0.37 & 0.62 & 432 & 402 \\
perold & 625 & 1376 & 608 & 546 & 549 & \textcolor{blue}{0.15d} & 0.08 & 0.05 & 0.02 & 38 & 27 \\
pilot.ja & 940 & 1988 & 823 & 680 & 675 & \textcolor{blue}{0.52d} & 0.32 & 0.09 & 0.10 & 97 & 33 \\
pilot & 1441 & 3652 & 1308 & 1284 & 1209 & \textcolor{blue}{10.19a} & 0.33 & 9.62 & 0.22 & 27 & 1628 \\
pilot4 & 410 & 1000 & 382 & 363 & 360 & \textcolor{blue}{0.03d} & 0.01 & 0.01 & 0.01 & 13 & 9 \\
pilot87 & 2030 & 4883 & 1976 & 1840 & 1800 & \textcolor{blue}{27.52a} & 2.48 & 24.68 & 0.34 & 111 & 2325 \\
pilotnov & 975 & 2172 & 1809 & 704 & 647 & \textcolor{blue}{1.80d} & 1.68 & 0.00 & 0.11 & 1154 & 2 \\
qap12 & 3192 & 8856 & 2980 & 2441 & 2449 & \textcolor{blue}{61.98d} & 27.69 & 8.81 & 25.45 & 377 & 222 \\
qap8 & 912 & 1632 & 656 & 466 & 440 & 1.05 & 0.28 & 0.19 & 0.57 & 56 & 103 \\
recipe & 91 & 180 & 24 & 24 & 24 & 0.00 & 0.00 & 0.00 & 0.00 & 2 & 5 \\
sc105 & 105 & 103 & 85 & 85 & 85 & 0.01 & 0.00 & 0.00 & 0.00 & 2 & 6 \\
sc205 & 205 & 203 & 184 & 184 & 184 & 0.01 & 0.00 & 0.00 & 0.00 & 2 & 6 \\
sc50a & 50 & 48 & 42 & 42 & 42 & 0.01 & 0.00 & 0.00 & 0.00 & 2 & 2 \\
sc50b & 50 & 48 & 48 & 48 & 48 & 0.00 & 0.00 & 0.00 & 0.00 & 2 & 2 \\
scagr25 & 471 & 500 & 317 & 307 & 307 & 0.01 & 0.00 & 0.00 & 0.00 & 2 & 2 \\
scagr7 & 129 & 140 & 98 & 97 & 97 & 0.00 & 0.00 & 0.00 & 0.00 & 2 & 0 \\
scfxm1 & 330 & 457 & 253 & 231 & 231 & 0.02 & 0.01 & 0.01 & 0.00 & 13 & 21 \\
scfxm2 & 660 & 914 & 515 & 473 & 471 & 0.06 & 0.03 & 0.02 & 0.01 & 24 & 27 \\
scfxm3 & 990 & 1371 & 754 & 711 & 711 & 0.11 & 0.03 & 0.04 & 0.04 & 35 & 33 \\
scorpion & 388 & 358 & 246 & 245 & 245 & 0.01 & 0.00 & 0.00 & 0.01 & 2 & 2 \\
scrs8 & 490 & 1169 & 276 & 276 & 276 & 0.02 & 0.00 & 0.00 & 0.01 & 2 & 10 \\
scsd1 & 77 & 760 & 31 & 7 & 11 & 0.03 & 0.01 & 0.02 & 0.00 & 13 & 61 \\
scsd6 & 147 & 1350 & 182 & 60 & 60 & 0.05 & 0.03 & 0.01 & 0.00 & 84 & 55 \\
scsd8 & 397 & 2750 & 551 & 147 & 208 & 0.14 & 0.12 & 0.01 & 0.00 & 175 & 21 \\
sctap1 & 300 & 480 & 263 & 164 & 163 & 0.06 & 0.04 & 0.02 & 0.00 & 102 & 46 \\
sctap2 & 1090 & 1880 & 817 & 562 & 561 & 0.26 & 0.11 & 0.06 & 0.09 & 231 & 143 \\
sctap3 & 1480 & 2480 & 1040 & 731 & 729 & 0.53 & 0.19 & 0.13 & 0.21 & 279 & 273 \\
seba & 515 & 1028 & 438 & 438 & 438 & 0.00 & 0.00 & 0.00 & 0.00 & 2 & 2 \\
share1b & 117 & 225 & 98 & 94 & 93 & \textcolor{blue}{0.00p} & 0.00 & 0.00 & 0.00 & 4 & 0 \\
share2b & 96 & 79 & 52 & 48 & 48 & 0.01 & 0.00 & 0.00 & 0.00 & 7 & 11 \\
shell & 536 & 1775 & 392 & 383 & 383 & 0.03 & 0.00 & 0.01 & 0.02 & 2 & 48 \\
ship04l & 402 & 2118 & 261 & 260 & 260 & 0.02 & 0.00 & 0.02 & 0.01 & 2 & 24 \\
ship04s & 402 & 1458 & 281 & 280 & 280 & 0.02 & 0.00 & 0.01 & 0.01 & 2 & 16 \\
ship08l & 778 & 4283 & 423 & 422 & 422 & 0.09 & 0.00 & 0.03 & 0.05 & 2 & 34 \\
ship08s & 778 & 2387 & 448 & 447 & 447 & 0.06 & 0.00 & 0.01 & 0.05 & 2 & 22 \\
ship12l & 1151 & 5427 & 707 & 706 & 706 & 0.21 & 0.00 & 0.04 & 0.16 & 2 & 57 \\
ship12s & 1151 & 2763 & 729 & 728 & 728 & 0.19 & 0.00 & 0.03 & 0.16 & 2 & 47 \\
sierra & 1227 & 2036 & 377 & 361 & 361 & 0.18 & 0.01 & 0.03 & 0.14 & 5 & 92 \\
stair & 356 & 467 & 350 & 349 & 349 & 0.01 & 0.00 & 0.00 & 0.00 & 2 & 2 \\
standata & 359 & 1075 & 71 & 50 & 50 & 0.05 & 0.00 & 0.04 & 0.00 & 10 & 146 \\
standgub & 361 & 1184 & 71 & 50 & 50 & 0.04 & 0.00 & 0.03 & 0.00 & 10 & 146 \\
standmps & 467 & 1075 & 190 & 174 & 174 & 0.04 & 0.00 & 0.03 & 0.00 & 15 & 143 \\
stocfor1 & 117 & 111 & 69 & 69 & 69 & 0.00 & 0.00 & 0.00 & 0.00 & 2 & 3 \\
stocfor2 & 2157 & 2031 & 1267 & 1267 & 1267 & 0.85 & 0.00 & 0.13 & 0.71 & 2 & 123 \\
truss & 1000 & 8806 & 802 & 691 & 691 & 0.90 & 0.22 & 0.56 & 0.12 & 102 & 300 \\
tuff & 333 & 587 & 167 & 122 & 122 & 0.04 & 0.02 & 0.02 & 0.00 & 43 & 60 \\
vtp.base & 198 & 203 & 55 & 55 & 55 & 0.00 & 0.00 & 0.00 & 0.00 & 2 & 2 \\
wood1p & 244 & 2594 & 39 & 39 & 39 & 0.18 & 0.00 & 0.18 & 0.00 & 2 & 134 \\
woodw & 1098 & 8405 & 696 & 550 & 550 & 0.61 & 0.24 & 0.23 & 0.13 & 146 & 81 \\

\bottomrule
\caption{Crossover results ($1e-6$ solutions) for the 100 NETLIB instances. ``p'' means the basis is only primal optimal (dual infeasible). ``d'' means the basis is only dual optimal (primal feasible). ``a'' means the basis is neither primal nor dual infeasible. All OLS subproblems are solved by direct QR factorization.} \label{tbl:netlib-1e-6}
\end{longtable}
}

{\footnotesize
\begin{longtable}{ |c|c|c|c|c|c|c|c|c|c|c|c|c| }

\toprule
 \multirow{2}{*}{prob} & \multirow{2}{*}{nRows} & \multirow{2}{*}{nCols} & \multicolumn{3}{c|}{\# supp.} & \multicolumn{5}{c|}{time (sec)} & \multicolumn{2}{c|}{nOLS} \\
\cline{4-13}
 & & & PDLP & COPT & cross & cross & aux. LP & primal & dual & lin. ind. & primal & dual \\
\midrule

\endfirsthead

\toprule
 \multirow{2}{*}{prob} & \multirow{2}{*}{nRows} & \multirow{2}{*}{nCols} & \multicolumn{3}{c|}{\# supp.} & \multicolumn{5}{c|}{time (sec)} & \multicolumn{2}{c|}{nOLS} \\
\cline{4-13}
 & & & PDLP & COPT & cross & cross & aux. LP & primal & dual & lin. ind. & primal & dual \\
\midrule

\endhead
\endfoot
25fv47 & 821 & 1571 & 600 & 583 & 584 & 1.01 & 0.96 & 0.01 & 0.00 & 0.03 & 1 & 2 \\
80bau3b & 2262 & 9799 & 1851 & 1758 & 1755 & 0.77 & 0.05 & 0.00 & 0.00 & 0.71 & 1 & 1 \\
adlittle & 56 & 97 & 61 & 45 & 44 & 0.01 & 0.01 & 0.00 & 0.00 & 0.00 & 1 & 1 \\
afiro & 27 & 32 & 14 & 13 & 14 & 0.00 & 0.00 & 0.00 & 0.00 & 0.00 & 1 & 1 \\
agg & 488 & 163 & 62 & 57 & 57 & 0.01 & 0.01 & 0.00 & 0.00 & 0.00 & 1 & 1 \\
agg2 & 516 & 302 & 141 & 120 & 121 & 0.01 & 0.01 & 0.00 & 0.00 & 0.00 & 2 & 1 \\
agg3 & 516 & 302 & 144 & 124 & 125 & 0.01 & 0.01 & 0.00 & 0.00 & 0.00 & 2 & 1 \\
bandm & 305 & 472 & 306 & 294 & 294 & 0.07 & 0.06 & 0.00 & 0.00 & 0.00 & 1 & 1 \\
beaconfd & 173 & 262 & 89 & 89 & 89 & 0.01 & 0.00 & 0.00 & 0.00 & 0.00 & 1 & 1 \\
blend & 74 & 83 & 56 & 54 & 54 & 0.01 & 0.00 & 0.00 & 0.00 & 0.00 & 1 & 1 \\
bnl1 & 643 & 1175 & 689 & 451 & 448 & 6.29 & 6.27 & 0.00 & 0.00 & 0.02 & 1 & 1 \\
bnl2 & 2324 & 3489 & 1519 & 1171 & 1168 & 2.22 & 1.19 & 0.00 & 0.01 & 1.01 & 2 & 3 \\
boeing1 & 351 & 384 & 207 & 196 & 195 & 0.04 & 0.03 & 0.00 & 0.00 & 0.00 & 1 & 3 \\
boeing2 & 166 & 143 & 69 & 55 & 55 & 0.01 & 0.01 & 0.00 & 0.00 & 0.00 & 1 & 3 \\
bore3d & 233 & 315 & 126 & 126 & 126 & 0.01 & 0.01 & 0.00 & 0.00 & 0.00 & 1 & 1 \\
brandy & 220 & 249 & 135 & 134 & 134 & 0.05 & 0.04 & 0.00 & 0.00 & 0.00 & 1 & 1 \\
capri & 271 & 353 & 246 & 220 & 225 & 0.02 & 0.02 & 0.00 & 0.00 & 0.00 & 1 & 1 \\
cre-a & 3516 & 4067 & 579 & 492 & 484 & 1.27 & 0.08 & 0.00 & 0.35 & 0.81 & 1 & 41 \\
cre-c & 3068 & 3678 & 590 & 508 & 503 & 2.22 & 0.22 & 0.00 & 1.23 & 0.75 & 2 & 46 \\
cycle & 1903 & 2857 & 994 & 224 & 105 & \textcolor{blue}{2.04p} & 2.00 & 0.00 & 0.01 & 0.03 & 1 & 9 \\
czprob & 929 & 3523 & 924 & 866 & 866 & 0.06 & 0.03 & 0.00 & 0.00 & 0.03 & 1 & 1 \\
d2q06c & 2171 & 5167 & 1612 & 1543 & 1541 & 93.46 & 92.75 & 0.00 & 0.05 & 0.65 & 1 & 11 \\
d6cube & 415 & 6184 & 136 & 115 & 101 & 13.38 & 13.30 & 0.00 & 0.06 & 0.01 & 1 & 7 \\
degen2 & 444 & 534 & 207 & 207 & 207 & 0.33 & 0.13 & 0.00 & 0.16 & 0.03 & 1 & 111 \\
degen3 & 1503 & 1818 & 640 & 637 & 637 & 8.20 & 2.87 & 0.01 & 4.41 & 0.79 & 1 & 475 \\
e226 & 223 & 282 & 127 & 127 & 127 & 0.05 & 0.04 & 0.00 & 0.00 & 0.00 & 1 & 1 \\
etamacro & 400 & 688 & 292 & 271 & 272 & 0.03 & 0.01 & 0.00 & 0.00 & 0.01 & 1 & 12 \\
fffff800 & 524 & 854 & 357 & 312 & 315 & 0.03 & 0.01 & 0.00 & 0.00 & 0.01 & 7 & 1 \\
finnis & 497 & 614 & 259 & 227 & 229 & 0.02 & 0.01 & 0.00 & 0.00 & 0.01 & 2 & 4 \\
fit1d & 24 & 1026 & 12 & 12 & 12 & 0.01 & 0.00 & 0.00 & 0.00 & 0.00 & 1 & 0 \\
fit1p & 627 & 1677 & 634 & 627 & 627 & 0.60 & 0.59 & 0.00 & 0.00 & 0.00 & 1 & 0 \\
fit2d & 25 & 10500 & 22 & 20 & 20 & 0.01 & 0.01 & 0.00 & 0.00 & 0.00 & 1 & 0 \\
fit2p & 3000 & 13525 & 3004 & 2997 & 2997 & 5.08 & 5.05 & 0.00 & 0.00 & 0.03 & 1 & 1 \\
forplan & 161 & 421 & 83 & 83 & 83 & 1.65 & 1.64 & 0.00 & 0.00 & 0.00 & 1 & 1 \\
ganges & 1309 & 1681 & 1278 & 1175 & 1171 & t &  &  &  &  &  &  \\
gfrd-pnc & 616 & 1092 & 349 & 336 & 336 & 0.05 & 0.02 & 0.00 & 0.01 & 0.02 & 1 & 23 \\
greenbeb & 2392 & 5405 & 1048 & 937 & 942 & 114.78 & 113.76 & 0.00 & 0.00 & 1.01 & 1 & 1 \\
grow15 & 300 & 645 & 533 & 299 & 300 & 0.20 & 0.19 & 0.00 & 0.00 & 0.00 & 1 & 0 \\
grow22 & 440 & 946 & 849 & 440 & 440 & 0.38 & 0.37 & 0.00 & 0.00 & 0.00 & 1 & 0 \\
grow7 & 140 & 301 & 237 & 140 & 140 & 0.02 & 0.02 & 0.00 & 0.00 & 0.00 & 1 & 0 \\
israel & 174 & 142 & 80 & 70 & 68 & 0.01 & 0.01 & 0.00 & 0.00 & 0.00 & 7 & 0 \\
kb2 & 43 & 41 & 27 & 27 & 27 & 0.01 & 0.01 & 0.00 & 0.00 & 0.00 & 1 & 0 \\
ken-07 & 2426 & 3602 & 2236 & 2234 & 2234 & 1.22 & 0.03 & 0.00 & 0.00 & 1.18 & 1 & 4 \\
lotfi & 153 & 308 & 126 & 99 & 98 & 0.01 & 0.01 & 0.00 & 0.00 & 0.00 & 1 & 1 \\
maros-r7 & 3136 & 9408 & 3136 & 3136 & 3136 & 0.33 & 0.28 & 0.01 & 0.00 & 0.01 & 1 & 0 \\
maros & 846 & 1443 & 345 & 340 & 339 & 1.84 & 1.77 & 0.00 & 0.01 & 0.05 & 3 & 11 \\
modszk1 & 687 & 1620 & 666 & 666 & 666 & 0.05 & 0.01 & 0.00 & 0.00 & 0.03 & 1 & 1 \\
nesm & 662 & 2923 & 726 & 550 & 545 & t &  &  &  &  &  &  \\
osa-07 & 1118 & 23949 & 357 & 355 & 355 & 0.16 & 0.13 & 0.00 & 0.00 & 0.02 & 1 & 1 \\
osa-14 & 2337 & 52460 & 873 & 781 & 781 & f &  &  &  &  &  &  \\
osa-30 & 4350 & 100024 & 1733 & 1536 & 1536 & f &  &  &  &  &  &  \\
pds-02 & 2953 & 7535 & 1379 & 332 & 320 & 0.40 & 0.06 & 0.00 & 0.02 & 0.31 & 3 & 30 \\
perold & 625 & 1376 & 580 & 546 & 547 & \textcolor{blue}{54.13d} & 54.10 & 0.00 & 0.00 & 0.02 & 1 & 1 \\
pilot.ja & 940 & 1988 & 789 & 681 & 661 & t &  &  &  &  &  &  \\
pilot & 1441 & 3652 & 1299 & 1287 & 1298 & t &  &  &  &  &  &  \\
pilot4 & 410 & 1000 & 374 & 364 & 364 & 99.07 & 99.06 & 0.00 & 0.00 & 0.01 & 2 & 1 \\
pilot87 & 2030 & 4883 & 1856 & 1840 & 1840 & t &  &  &  &  &  &  \\
pilotnov & 975 & 2172 & 1809 & 708 & 642 & \textcolor{blue}{54.40d} & 54.26 & 0.03 & 0.00 & 0.10 & 11 & 1 \\
qap12 & 3192 & 8856 & 2940 & 2462 & 2446 & 63.12 & 32.86 & 0.06 & 4.88 & 25.29 & 1 & 112 \\
qap8 & 912 & 1632 & 656 & 466 & 418 & 0.53 & 0.09 & 0.00 & 0.16 & 0.26 & 1 & 96 \\
recipe & 91 & 180 & 24 & 24 & 24 & 0.01 & 0.00 & 0.00 & 0.00 & 0.00 & 1 & 1 \\
sc105 & 105 & 103 & 85 & 85 & 85 & 0.01 & 0.00 & 0.00 & 0.00 & 0.00 & 1 & 1 \\
sc205 & 205 & 203 & 184 & 184 & 184 & 0.01 & 0.01 & 0.00 & 0.00 & 0.00 & 1 & 1 \\
sc50a & 50 & 48 & 42 & 42 & 42 & 0.01 & 0.00 & 0.00 & 0.00 & 0.00 & 1 & 1 \\
sc50b & 50 & 48 & 48 & 48 & 48 & 0.01 & 0.00 & 0.00 & 0.00 & 0.00 & 1 & 1 \\
scagr25 & 471 & 500 & 317 & 307 & 307 & 0.01 & 0.01 & 0.00 & 0.00 & 0.00 & 1 & 1 \\
scagr7 & 129 & 140 & 98 & 97 & 97 & 0.00 & 0.00 & 0.00 & 0.00 & 0.00 & 1 & 0 \\
scfxm1 & 330 & 457 & 248 & 231 & 232 & 0.03 & 0.02 & 0.00 & 0.00 & 0.00 & 2 & 2 \\
scfxm2 & 660 & 914 & 501 & 473 & 471 & 0.06 & 0.04 & 0.00 & 0.00 & 0.01 & 1 & 3 \\
scfxm3 & 990 & 1371 & 754 & 711 & 711 & 0.15 & 0.10 & 0.00 & 0.00 & 0.04 & 1 & 1 \\
scorpion & 388 & 358 & 245 & 245 & 245 & 0.02 & 0.01 & 0.00 & 0.00 & 0.01 & 1 & 1 \\
scrs8 & 490 & 1169 & 276 & 276 & 276 & 0.05 & 0.03 & 0.00 & 0.00 & 0.01 & 1 & 1 \\
scsd1 & 77 & 760 & 31 & 7 & 11 & 0.02 & 0.01 & 0.00 & 0.01 & 0.00 & 1 & 59 \\
scsd6 & 147 & 1350 & 182 & 63 & 56 & 0.02 & 0.01 & 0.00 & 0.01 & 0.00 & 1 & 34 \\
scsd8 & 397 & 2750 & 551 & 147 & 143 & 0.16 & 0.14 & 0.00 & 0.01 & 0.00 & 1 & 21 \\
sctap1 & 300 & 480 & 263 & 164 & 168 & 0.04 & 0.03 & 0.00 & 0.00 & 0.01 & 4 & 5 \\
sctap2 & 1090 & 1880 & 811 & 562 & 564 & 0.13 & 0.03 & 0.00 & 0.00 & 0.09 & 3 & 5 \\
sctap3 & 1480 & 2480 & 1038 & 731 & 732 & 0.26 & 0.03 & 0.00 & 0.01 & 0.21 & 3 & 21 \\
seba & 515 & 1028 & 438 & 438 & 438 & 0.01 & 0.01 & 0.00 & 0.00 & 0.00 & 1 & 1 \\
share1b & 117 & 225 & 95 & 94 & 94 & 0.03 & 0.03 & 0.00 & 0.00 & 0.00 & 1 & 0 \\
share2b & 96 & 79 & 52 & 48 & 48 & 0.04 & 0.03 & 0.00 & 0.00 & 0.00 & 1 & 2 \\
shell & 536 & 1775 & 391 & 383 & 383 & 0.03 & 0.01 & 0.00 & 0.00 & 0.02 & 1 & 2 \\
ship04l & 402 & 2118 & 261 & 260 & 260 & 0.03 & 0.02 & 0.00 & 0.00 & 0.01 & 1 & 1 \\
ship04s & 402 & 1458 & 281 & 280 & 280 & 0.03 & 0.02 & 0.00 & 0.00 & 0.01 & 1 & 4 \\
ship08l & 778 & 4283 & 422 & 422 & 422 & 0.07 & 0.02 & 0.00 & 0.00 & 0.05 & 1 & 1 \\
ship08s & 778 & 2387 & 447 & 447 & 447 & 0.07 & 0.02 & 0.00 & 0.00 & 0.05 & 1 & 1 \\
ship12l & 1151 & 5427 & 707 & 706 & 706 & 0.20 & 0.03 & 0.00 & 0.00 & 0.16 & 1 & 1 \\
ship12s & 1151 & 2763 & 728 & 728 & 728 & 0.18 & 0.02 & 0.00 & 0.00 & 0.16 & 1 & 1 \\
sierra & 1227 & 2036 & 373 & 361 & 361 & 0.16 & 0.02 & 0.00 & 0.00 & 0.14 & 1 & 7 \\
stair & 356 & 467 & 350 & 349 & 349 & 0.20 & 0.20 & 0.00 & 0.00 & 0.00 & 1 & 1 \\
standata & 359 & 1075 & 71 & 50 & 50 & 0.01 & 0.01 & 0.00 & 0.00 & 0.00 & 1 & 1 \\
standgub & 361 & 1184 & 71 & 50 & 50 & 0.01 & 0.01 & 0.00 & 0.00 & 0.00 & 1 & 6 \\
standmps & 467 & 1075 & 190 & 174 & 174 & 0.02 & 0.01 & 0.00 & 0.00 & 0.00 & 1 & 6 \\
stocfor1 & 117 & 111 & 69 & 69 & 69 & 0.01 & 0.01 & 0.00 & 0.00 & 0.00 & 1 & 1 \\
stocfor2 & 2157 & 2031 & 1267 & 1267 & 1267 & 0.76 & 0.02 & 0.00 & 0.02 & 0.72 & 1 & 17 \\
truss & 1000 & 8806 & 802 & 691 & 689 & 0.73 & 0.08 & 0.00 & 0.47 & 0.12 & 1 & 262 \\
tuff & 333 & 587 & 165 & 122 & 113 & 0.64 & 0.64 & 0.00 & 0.00 & 0.00 & 1 & 1 \\
vtp.base & 198 & 203 & 55 & 55 & 55 & 0.01 & 0.00 & 0.00 & 0.00 & 0.00 & 1 & 1 \\
wood1p & 244 & 2594 & 39 & 39 & 39 & 1.40 & 1.40 & 0.00 & 0.00 & 0.00 & 1 & 1 \\
woodw & 1098 & 8405 & 696 & 550 & 550 & 0.19 & 0.05 & 0.00 & 0.00 & 0.13 & 1 & 1 \\

\bottomrule
\caption{Crossover results with auxiliary LPs ($1e-8$ solutions) for the 100 NETLIB instances. ``p'' means the basis is only primal optimal (dual infeasible). ``d'' means the basis is only dual optimal (primal infeasible). ``a'' means the basis is near-optimal but primal-dual infeasible. ``t'' means crossover timed out. ``f'' means crossover failed. All OLS subproblems are solved by direct QR factorization.} \label{tbl:netlib-lp-1e-8}
\end{longtable}
}

\end{document}